\documentclass[11pt]{article}

\usepackage[centertags]{amsmath}
\usepackage{amsfonts}
\usepackage{amssymb}
\usepackage{amsthm}
\usepackage{newlfont}
\usepackage{easybmat}
\usepackage{bibspacing}
\usepackage{verbatim}

\DeclareMathOperator{\arccot}{arccot}
\DeclareMathOperator{\arccoth}{arccoth}
\DeclareMathOperator{\arctanh}{arctanh}

\newtheorem{thm}{Theorem}[section]
\newtheorem*{thm*}{Theorem}
\newtheorem{cor}[thm]{Corollary}

\newtheorem{lem}[thm]{Lemma}

\newtheorem{prop}[thm]{Proposition}
\newtheorem*{prop*}{Proposition}

\newtheorem*{conj*}{Conjecture}
\newtheorem{defn}[thm]{Definition}
\newtheorem*{defn*}{Definition}
\theoremstyle{definition}
\newtheorem{rem}[thm]{\textbf{Remark}}
\newtheorem*{rmk*}{Remark}
\newtheorem*{fact*}{Fact}

\theoremstyle{proof}

\newcommand{\CD}{\text{CD}}
\newcommand{\CDD}{\text{CDD}}
\newcommand{\II}{\text{II}}
\newcommand{\sign}{\text{sign}}
\newcommand{\med}{\text{med}}
\newcommand{\LogHess}{\text{LogHess}}
\newcommand{\Ric}{\text{Ric}}
\newcommand{\norm}[1]{\left\Vert#1\right\Vert}
\newcommand{\snorm}[1]{\Vert#1\Vert}
\newcommand{\abs}[1]{\left\vert#1\right\vert}
\newcommand{\set}[1]{\left\{#1\right\}}
\newcommand{\brac}[1]{\left(#1\right)}
\newcommand{\scalar}[1]{\left \langle #1 \right \rangle}

\newcommand{\Real}{\mathbb{R}}

\newcommand{\eps}{\varepsilon}
\newcommand{\K}{\mathcal{K}}
\newcommand{\I}{\mathcal{I}}
\newcommand{\N}{\mathcal{N}}
\newcommand{\GL}{\mathcal{GL}^\flat}
\newcommand{\J}{\mathcal{I}^\flat}

\newcommand{\vol}{\textrm{Vol}}

\newlength{\defbaselineskip}
\newcommand{\setlinespacing}[1]           {\setlength{\baselineskip}{#1 \defbaselineskip}}

\numberwithin{equation}{section}

\begin{document}

\title{Beyond traditional Curvature-Dimension I: new model spaces for isoperimetric and concentration inequalities in negative dimension}
\author{Emanuel Milman\textsuperscript{1}}

\date{}

\footnotetext[1]{Department of Mathematics,
Technion - Israel Institute of Technology, Haifa 32000, Israel. Supported by ISF (grant no. 900/10), BSF (grant no. 2010288) and Marie-Curie Actions (grant no. PCIG10-GA-2011-304066).
Email: emilman@tx.technion.ac.il.\\
Mathematics Subject Classification (2010): 32F32, 53C21, 39B62, 58J50.} 
\maketitle

\begin{abstract}
We study the isoperimetric, functional and concentration properties of $n$-dimensional weighted Riemannian manifolds satisfying the Curvature-Dimension condition, when the generalized dimension $N$ is negative, and more generally, is in the range $N \in (-\infty,1)$, extending the scope from the traditional range $N \in [n,\infty]$. In particular, we identify the correct one-dimensional model-spaces under an additional diameter upper bound, and discover a new case yielding a \emph{single} model space (besides the previously known $N$-sphere and Gaussian measure when $N \in [n,\infty]$): a (positively curved) sphere of (possibly negative) dimension $N \in (-\infty,1)$. 
When curvature is non-negative, we show that arbitrarily weak concentration implies an $N$-dimensional Cheeger isoperimetric inequality, and derive various weak Sobolev and Nash-type inequalities on such spaces. When curvature is strictly positive, we observe that such spaces satisfy a Poincar\'e inequality uniformly for all $N \in (-\infty,1-\eps]$, and enjoy a two-level concentration of the type $\exp(-\min(t,t^2))$. Our main technical tool is a generalized version of the Heintze--Karcher theorem, which we extend to the range $N \in (-\infty,1)$.
\end{abstract}

\section{Introduction}

Let $(M^n,g)$ denote an $n$-dimensional ($n \geq 2$) complete connected oriented smooth Riemannian manifold with (possibly empty) boundary, and let $\mu$ denote a measure on $M$ having density $\Psi$ with respect to the Riemannian volume form $vol_g$. We assume that $M$ is geodesically convex (any two points may be connected by a distance minimizing geodesic), that $\partial M$ is $C^2$ smooth, and that $\Psi$ is positive and $C^2$ smooth on the entire $M$ (all the way up to the boundary). As usual, we denote by $Ric_g$ the Ricci curvature tensor and by $\nabla_g$ the Levi-Civita covariant derivative.

\begin{defn*}[Generalized Ricci Tensor]
Given $N \in (-\infty,\infty]$, the $N$-dimensional generalized Ricci curvature tensor $Rig_{g,\mu,N}$ is defined as:
\begin{equation} \label{eq:Ric-Tensor}
Ric_{g,\mu,N}  :=  Ric_g - \LogHess_{N-n} \Psi ,
\end{equation}
where:
\[
\LogHess_\N \Psi := \nabla^2_g \log(\Psi) + \frac{1}{\N} \nabla_g \log(\Psi) \otimes \nabla_g \log(\Psi) =  \N \frac{\nabla^2_g \Psi^{\frac{1}{\N}}}{\Psi^{\frac{1}{\N}}} .
\]
To make sense of the latter tensor when $\N \in \set{0,\infty}$, we employ throughout the convention that $\frac{1}{\infty} = 0$, $\frac{1}{0} = +\infty$ and $\infty \cdot 0 = 0$. 
\end{defn*}

\begin{defn*}[Curvature-Dimension and Curvature-Dimension-Diameter conditions]
$(M^n,g,\mu)$ satisfies the Curvature-Dimension condition $\CD(\rho,N)$ ($\rho \in \Real$ and $N \in (-\infty,\infty]$) if $\Ric_{g,\mu,N} \geq \rho g$ as symmetric $2$-tensors on $M$. It satisfies the Curvature-Dimension-Diameter condition $\CDD(\rho,N,D)$ if in addition its diameter is bounded above by $D \in (0,\infty]$. 
\end{defn*}

Note that $\CD(\rho,N)$ is satisfied with $N=n$ if and only if $\Psi$ is constant and $Ric_{g,\mu,n} = Ric_g \geq \rho g$ (the classical constant density case). The generalized Ricci tensor (\ref{eq:Ric-Tensor}) was introduced with $N=\infty$ by Lichnerowicz \cite{Lichnerowicz1970GenRicciTensorCRAS,Lichnerowicz1970GenRicciTensor} and in general by Bakry \cite{BakryStFlour} (cf. Lott \cite{LottRicciTensorProperties}). The Curvature-Dimension condition was introduced by Bakry and \'Emery in equivalent form in \cite{BakryEmery} (in the more abstract framework of diffusion generators) - see Subsection \ref{subsec:BE} for a discussion. Its name stems from the fact that the generalized Ricci tensor incorporates information on curvature and dimension from both the geometry of $(M,g)$ and the measure $\mu$, and so $\rho$ may be thought of as a generalized-curvature lower bound, and $N$ as a generalized-dimension upper bound. The $\CD(\rho,N)$ condition has been an object of extensive study over the last two decades (see e.g. also \cite{QianWeightedVolumeThms,LedouxLectureNotesOnDiffusion,CMSInventiones, CMSManifoldWithDensity, VonRenesseSturmRicciChar,BakryQianGenRicComparisonThms,WeiWylie-GenRicciTensor,MorganBook4Ed,EMilmanSharpIsopInqsForCDD,KolesnikovEMilmanReillyPart1} and the references therein), especially since Perelman's work on the Poincar\'e Conjecture \cite{PerelmanEntropyFormulaForRicciFlow}, and the extension of the Curvature-Dimension condition to the metric-measure space setting by Lott--Sturm--Villani \cite{SturmCD12,LottVillaniGeneralizedRicci}.

\medskip

So far, most of the activity has been in the range $N \in [n,\infty]$. Very recently, a growing interest in weighted manifolds with negative generalized dimension $N < 0$ has begun to emerge - see the works by Ohta--Takatsu \cite{OhtaTakatsu-GenEntropyI,OhtaTakatsu-GenEntropyII}, Ohta \cite{Ohta-NegativeN}, Kolesnikov--Milman \cite{KolesnikovEMilmanReillyPart1} and Klartag \cite{KlartagLocalizationOnManifolds}. As we shall see, the case of negative generalized dimension is actually quite natural. In the Euclidean setting, measures $\mu$ on $(\Real^n,\abs{\cdot})$ satisfying the $\CD(0,N)$ condition with $\frac{1}{N} \in [-\infty,1/n]$ have already been studied by Borell \cite{BorellConvexMeasures} (cf. \cite{BrascampLiebPLandLambda1}), who dubbed them convex (or $\frac{1}{N}$-concave) measures - see \cite{BobkovConvexHeavyTailedMeasures, BobkovLedouxWeightedPoincareForHeavyTails, NguyenDimensionalBrascampLieb, EMilmanRotemHomogeneous,KolesnikovEMilmanReillyPart1} for a more detailed account of these measures and their useful properties. 

In this work, we focus on the case that $N < 0$ (in fact, more generally, $N \in (-\infty,1)$), and study the isoperimetric, functional and concentration properties of weighted manifolds satisfying the $\CD(\rho,N)$ condition when $N$ is in that range, for arbitrary curvature ($\rho \in \Real$) and upper bound on diameter, for non-negative curvature ($\rho=0$) and for positive curvature ($\rho > 0$). 

Finally, we mention the recent works by Wylie \cite{Wylie-SectionalCurvature,Wylie-CheegerGromoll} and Kennard--Wylie \cite{KennardWylie-WeightedSectionalCurvature} on generalizing the notion of \emph{sectional curvature} to the weighted manifold setting and its applications. Remarkably, whereas all of our results break down in the remaining range $N \in [1,n)$ (see Subsection \ref{subsec:no-extend}), these authors obtain various topological and geometric results even for the case $N=1$, foreshadowing additional future developments in this direction (see also \cite{EMilman-GradedCD}). 

\subsection{Jacobian Curvature-Dimension condition}

Given a $C^2$ smooth co-oriented hypersurface $S \subset (M^n,g)$ with normal unit vector field $\nu$, let $F_S : S \times \Real \rightarrow M$ denote the normal map given by $F_S(x,t) = \exp_x(t \nu(x))$ (strictly speaking, since $M$ may have a boundary, this is only defined for all $(x,t)$ so that $t \nu(x)$ is in the domain $\text{Dom}$ of $\exp_x$). Let $J_{S,x}(t)$ denote the Jacobian of this map, so that the pull-back of $vol_g$ by $F_S$ is given by $F_S^*(vol_g) = J_{S,x}(t) dvol_S(x) dt$, where $vol_S$ is the induced Riemannian volume on $S$.  
The \emph{weighted} Jacobian $J_{S,\mu,x}(t)$, which also takes into account the corresponding densities, is defined as:
\[
J_{S,\mu,x}(t) = J_{S,x}(t) \frac{\Psi(F_S(x,t))}{\Psi(x)} ,
\]
so that $F_S^*(\mu) = J_{S,\mu,x}(t) dvol_{S,\mu}(x) dt$, where $vol_{S,\mu} = \Psi \cdot vol_S$. 

\begin{defn*}[Jacobian Curvature-Dimension condition]
$(M^n,g,\mu)$ satisfies a Jacobian Curvature-Dimension condition $\text{Jac-CD}(\rho,N)$ ($\rho \in \Real$, $N \in (-\infty,\infty]$) if for any  $C^2$ smooth co-oriented hypersurface $S$ and any $x \in S$, the weighted Jacobian $J(t) = J_{S,\mu,x}(t)$ satisfies the following ordinary differential inequality on the maximal interval $L$ containing the origin on which it is defined and positive:
\begin{equation} \label{eq:JacEq}
-\LogHess_{N-1} J \geq \rho .
\end{equation}
In other words, $(L,\abs{\cdot},J(t) dt)$ satisfies $\CD(\rho,N)$ (see Subsection \ref{subsec:1D} for the extension of the Curvature-Dimension condition to the one-dimensional setting). 

\noindent
For concreteness, we record that our notation in the one-dimensional case becomes:
\[
\LogHess_{\N} J = (\log J)'' + \frac{1}{\N} ((\log J)')^2 = \N \frac{\brac{J^{1/\N}}''}{J^{1/\N}} ,
\]
with the usual interpretation when $\N \in \set{0,\infty}$. 
\end{defn*}

The generalized Heintze-Karcher theorem states that the $\CD(\rho,N)$ condition implies the $\text{Jac-CD}(\rho,N)$ condition, when $N \in [n,\infty]$. The classical case $N=n$ (constant density) is due to Heintze and Karcher \cite{HeintzeKarcher}; this was extended to the weighted manifold setting by Bayle \cite[Appendix E]{BayleThesis} ($N \in (n,\infty)$) and Morgan \cite{MorganManifoldsWithDensity} ($N=\infty$). 
In Section \ref{sec:HK}, we provide our own version of the proof and extend it to the entire range $N \in (-\infty,1) \cup [n,\infty]$. 

\begin{thm}[Generalized Heintze-Karcher Theorem, extending \cite{HeintzeKarcher,BayleThesis,MorganManifoldsWithDensity}] 
The $\CD(\rho,N)$ condition implies the $\text{Jac-CD}(\rho,N)$ for all $\rho \in \Real$ and $N \in (-\infty,1) \cup [n,\infty]$.
\end{thm}

\noindent Our entire analysis in this work is based on this extension. 

\subsection{Extremal Jacobian Solutions}

Given $\delta \in \Real$, set as usual:
\[
\begin{array}{ccc}
 s_\delta(t) := \begin{cases}
\sin(\sqrt{\delta} t)/\sqrt{\delta} & \delta > 0 \\
t & \delta = 0 \\
\sinh(\sqrt{-\delta} t)/\sqrt{-\delta} & \delta < 0
\end{cases}

& , &

 c_\delta(t) := \begin{cases}
\cos(\sqrt{\delta} t) & \delta > 0 \\
1 & \delta = 0 \\
\cosh(\sqrt{-\delta} t) & \delta < 0
\end{cases}
\end{array} ~.
\]

Given a continuous function $f : \Real \rightarrow \Real$ with $f(0) \geq 0$, we denote by $f_+ : \Real \rightarrow \Real_+$ the function coinciding with $f$ between its first non-positive and first positive roots, and vanishing everywhere else, i.e. $f_+ := f 1_{[\xi_{-},\xi_{+}]}$ with $\xi_{-} = \sup\set{\xi \leq 0; f(\xi) = 0}$ and $\xi_{+} = \inf\set{\xi > 0; f(\xi) = 0}$.

\begin{defn*} Given $H, \rho \in \Real$ and $N \in (-\infty,\infty]$, set  $\delta := \rho / (N-1)$ if $N \neq 1$ and define the following (Jacobian) function $\Real \ni t \mapsto J_{H,\rho,N}(t) \in [0,\infty]$:
\[
J_{H,\rho,N}(t) :=
\begin{cases}
\brac{\brac{c_\delta(t) + \frac{H}{N-1} s_\delta(t)}_+}^{N-1} & N \notin \set{1,\infty} \\
\exp(H t - \frac{\rho}{2} t^2) & N = \infty \\
 1 & N = 1 , \rho = 0  \\
 \infty & \text{otherwise}
\end{cases} ~.
\] 
\end{defn*}

\begin{rem} \label{rem:J-eq} Observe that when $N \neq 1$,  
$J_{H,\rho,N}$ coincides (with the usual interpretation when $N =\infty$) with the solution $J$ to the following second order ODE 
on the maximal interval containing the origin where such a solution exists: \[
-\LogHess_{N-1} J = \rho ~,~ J(0) = 1 ~,~ J'(0) = H ~.
\]
Also observe that since $c_\delta(t) = 1 - \frac{\delta}{2} t^2 + o(\delta)$ and $s_\delta(t) = t + o(\delta)$ as $\delta \rightarrow 0$, it follows that $\lim_{N \rightarrow \infty} J_{H,\rho,N} =  J_{H,\rho,\infty}$.
These Jacobian functions naturally appear in comparison theorems on weighted Riemannian manifolds satisfying a Curvature-Dimension condition, such as the generalized Heintze-Karcher theorem. The parameters $H$, $\rho$, $N$ and $\delta$ may therefore be interpreted as (generalized) mean-curvature, Ricci curvature lower bound, dimension upper bound, and average sectional curvature lower bound, respectively.  
Our convention when $N=1$ is purely for consistency in pathological cases, and we could have also defined $J_{H,\rho,1} = \lim_{N \rightarrow 1+} J_{H,\rho,N}$ (since $N$ is an \emph{upper} bound on the generalized dimension). Note that our convention for $J_{H,\rho,N}$ is slightly different than the one employed in \cite{EMilmanSharpIsopInqsForCDD}.
\end{rem}

\subsection{Curvature-Dimension-Diameter Isoperimetric Inequalities}

Let $(\Omega,d)$ denote a separable metric space, and let $\mu$ denote a Borel measure on $(\Omega,d)$. The Minkowski (exterior) boundary measure $\mu^+(A)$ of a Borel set $A \subset \Omega$ is defined as $\mu^+(A) := \liminf_{\eps \to 0} \frac{\mu(A^d_{\eps} \setminus A)}{\eps}$, where $A_{\eps}=A^d_{\eps} := \set{x \in \Omega ; \exists y
\in A \;\; d(x,y) < \eps}$ denotes the $\eps$ extension of $A$ with
respect to the metric $d$. We assume henceforth that $\mu$ is a probability measure. The isoperimetric profile $\I =
\I(\Omega,d,\mu)$ is defined as the pointwise maximal function $\I
: [0,1] \rightarrow \Real_+ \cup \set{+\infty}$, so that $\mu^+(A) \geq \I(\mu(A))$, for
all Borel sets $A \subset \Omega$. An isoperimetric inequality measures the relation between the boundary measure and the measure of a set, by providing a lower bound on $\I(\Omega,d,\mu)$ by some (non-trivial) function $I: [0,1] \rightarrow \Real_+$. In our manifold-with-density setting, we will always assume that the metric $d$ is given by the induced geodesic distance on $(M,g)$, and write $\I = \I(M,g,\mu)$.

When $(\Omega,d) = (\Real,|\cdot|)$, we also define $\J  = \J(\Real,\abs{\cdot},\mu)$ as the pointwise maximal function $\J : [0,1] \rightarrow \Real_+ \cup \set{+\infty}$, so that $\mu^+(A) \geq \J(\mu(A))$ for all half lines $A = (-\infty,a)$ and $A = (a,\infty)$ (the difference with the function $\I$ being that the latter is tested on arbitrary Borel sets $A$). Obviously $\J \geq \I$, and a result of S. Bobkov \cite[Proposition 2.1]{BobkovExtremalHalfSpaces} asserts that $\J = \I$ when $\mu = f(x) dx$ and $f$ is \emph{log-concave}, meaning that $- \log(f) : \Real \rightarrow \Real \cup \set{+\infty}$ is convex.

Finally, given a Borel measure $\eta$ on $\Real$ and an interval $L \subset \Real$, we denote $\eta_L = \eta|_L / \eta(L)$ if $\eta(L) \in (0,\infty)$, and set $\I(\eta,L) := \I(L,\abs{\cdot},\eta_{L})$ and $\J(\eta,L) := \J(L,\abs{\cdot},\eta_{L})$. When $\eta(L) = 0$ we set $\J(\eta,L) = \I(\eta,L) \equiv +\infty$, and when $\eta(L) = \infty$ we set $\J(\eta,L) = \I(\eta,L) \equiv 0$. When $\eta_f = f(x) dx$, we denote for short $\I(f,L) = \I(\eta_f,L)$ and $\J(f,L) = \J(\eta_f,L)$.

 \begin{defn*}[Curvature-Dimension-Diameter Isoperimetric Profiles] Given $\rho \in \Real$, $N \in (-\infty,\infty]$ and $D \in (0,\infty]$, we define the following Curvature-Dimension-Diameter (CDD) Isoperimetric Profiles:
\begin{enumerate}
\item The CDD Gromov--L\'evy Isoperimetric Profile $\GL_{\rho,N,D}$ is defined for all $v \in [0,1]$ as:
\begin{equation} \label{eq:GL-CD-Profile}
\GL_{\rho, N, D}(v) := \inf_{(a,b) \in \Delta_D, H \in \Real} \max \brac{\frac{v}{\int_{-a}^0 J_{H,\rho,N}(t) dt},\frac{1-v}{\int_0^b J_{H,\rho,N}(t) dt}} .
\end{equation}
For consistency, when $N = 1$ and $\rho = 0$, we modify its values at the end points to be $\GL_{0, 1, D}(0) = \GL_{0, 1, D}(1) = 0$.  
\item The CDD Flat Isoperimetric Profile $\J_{\rho, N, D}$ is defined for all $v \in [0,1]$ as:
\begin{equation} \label{eq:Flat-CD-Profile}
\J_{\rho, N, D}(v) := \inf_{(a,b) \in \Delta_D} \inf_{H \in \Real} \set{ \J\brac{ J_{H,\rho,N}, [-a,b] }(v) ;  \int_{-a}^b J_{H,\rho,N}(t) dt < \infty}.
\end{equation}
\end{enumerate}
Here:
\[
\Delta_D := \begin{cases} \set{ (a,b) \; ; \; a,b> 0 ~,~ a + b = D } & D < \infty \\
 \set{ (\infty,\infty) } & D = \infty \end{cases} ~,
\]
our convention is that $[-\infty,\infty] = \Real$, and the inner infimum in (\ref{eq:Flat-CD-Profile}) over an empty set of $H$'s is defined as $0$. When $D = \infty$, we simply write $\GL_{\rho, N} = \GL_{\rho, N, \infty}$ and $\J_{\rho,N} = \J_{\rho,N,\infty}$. 
\end{defn*}

It was shown in our previous work \cite{EMilmanSharpIsopInqsForCDD} that the $\text{Jac-CD}(\rho,N)$ condition together with the assumption that the diameter is bounded from above by $D \in (0,\infty]$,  imply the following isoperimetric inequality:
\begin{equation} \label{eq:CDD-II} \I(M^n,g,\mu)(v) \geq \GL_{\rho, N, D}(v) = \J_{\rho,N,D}(v) \;\;\; \forall v \in [0,1] ,
\end{equation}
for all $N \in [n,\infty]$ and $\rho \in \Real$; in particular, the Gromov--L\'evy and Flat Isoperimetric profiles coincide in that range. Employing the generalized Heintze--Karcher theorem, the same conclusion was deduced under the $\text{CDD}(\rho,N,D)$ condition in the above range of parameters. All of our definitions naturally extend to the one-dimensional case $n=1$ (see Subsection \ref{subsec:1D}), wherein (\ref{eq:CDD-II}) is also immediate to verify under the $\text{CDD}(\rho,N,D)$ condition (see \cite[Corollary 3.2]{EMilmanSharpIsopInqsForCDD} for the case $N \in (1,\infty]$; the case $N=1$ is trivial). 
Furthermore, it was shown in \cite{EMilmanSharpIsopInqsForCDD} that (\ref{eq:CDD-II}) is sharp for the entire range of these parameters, for all $v \in [0,1]$ and $n \geq 2$. The analogous sharpness in the case $n=1$ is immediate since the one-dimensional spaces $([-a,b],\abs{\cdot},c J_{H,\rho,N}(t) dt)$ all satisfy the $\CDD(\rho,N,D)$ condition (with the exception of the pathological case $N=n=1$, $\rho < 0$ and $D < \infty$, whose model space is ill-defined, and which we henceforth exclude from any further discussion regarding sharpness). 

\subsection{Isoperimetric Inequalities for Arbitrary Curvature and Diameter}

Our first result extends these isoperimetric inequalities to the entire range $N \in (-\infty,1) \cup [n,\infty]$. 
In Subsection \ref{subsec:no-extend}, we comment that this extension is \emph{no longer valid} when $N \in [1,n)$. A subtle detail appearing below is that the Gromov-L\'evy and Flat Isoperimetric profiles may no longer coincide in the range $N \in (0,1)$ as in (\ref{eq:CDD-II}).

\begin{thm}[Curvature-Dimension-Diameter Isoperimetric Inequality, extending \cite{EMilmanSharpIsopInqsForCDD}] \label{thm:CDD-II} Let $(M^n,g,\mu)$ satisfy $\CDD(\rho,N,D)$ with $n \geq 1$, $\rho \in \Real$, $D \in (0,\infty]$ and $N \in (-\infty,\infty]$. Then the following isoperimetric inequalities hold:
\begin{enumerate}
\item If $N \in (-\infty,1) \cup [n,\infty]$ then:
\[
\I(M,g,\mu) \geq \GL_{\rho,N,D} .
\]
In particular, the case $n=1$ yields $\J_{\rho,N,D} \geq \GL_{\rho,N,D}$ for all $\rho \in \Real$, $N \in (-\infty,\infty]$ and $D \in (0,\infty]$.
\item
When $D = \infty$ or $N \in (-\infty,0] \cup [1,\infty]$, then $\GL_{\rho,N,D} = \J_{\rho,N,D}$.
Consequently, when $D = \infty$ and $N \in (0,1)$ or when $N \in (-\infty,0] \cup [n,\infty]$, then:
\[
\I(M,g,\mu) \geq \J_{\rho,N,D} .
\]
\end{enumerate}
\end{thm}

As in \cite{EMilmanSharpIsopInqsForCDD}, the $\CDD(\rho,N,D)$ isoperimetric inequality was deliberately formulated to cover the entire range of values for $\rho$, $N$ and $D$ simultaneously, indicating its universal character, but it may be easily reformulated in an equivalent simplified manner, depending on the different values of these parameters. When $N \in [n,\infty]$, seven cases were described in \cite{EMilmanSharpIsopInqsForCDD}, extending the classical sharp isoperimetric comparison theorems of Gromov--L\'evy \cite{GromovGeneralizationOfLevy} (extended by Bayle \cite[Appendix E]{BayleThesis}) and Bakry--Ledoux \cite{BakryLedoux}. Here, we add four more cases corresponding to the range $N \in (-\infty,1)$, revealing new model spaces for the isoperimetric problem under a Curvature-Dimension-Diameter condition in that range:

\begin{cor} \label{cor:model}
The $\CDD(\rho,N,D)$ condition implies the following new isoperimetric inequalities in the range $N \in (-\infty,1)$ (all infima and minima below are interpreted pointwise, and $\delta := \rho / (N-1)$):
\begin{description}
 \item[Case 1 - \mbox{$N \in (-\infty,1)$}, $\rho > 0$, $D=\infty$:]
\[
\I(M^n,g,\mu) \geq \J(\cosh(\sqrt{-\delta} t)^{N-1}, \Real) . \]
\item[Case 2 - \mbox{$N \in (-\infty,0]$}, $\rho > 0$, $D < \infty$:]
\[
 \I(M^n,g,\mu) \geq \min \left \{ \begin{array}{l}
\inf_{\xi > 0} \J( \sinh(\sqrt{-\delta} t)^{N-1}, [\xi,\xi+D] ) ~ , \\
\phantom{\inf_{\xi \in \Real}} \J( \;\; \exp(\sqrt{-\delta} t)^{N-1} , [0,D] ) ~, \\
\inf_{\xi \in \Real} \J(\cosh(\sqrt{-\delta} t)^{N-1},[\xi,\xi+D] )
\end{array}
\right \} .
\]
\item[Case 3 - \mbox{$N \in (-\infty,0]$}, $\rho = 0$, $D < \infty$:]
\[
\I(M^n,g,\mu) \geq 
\min \left \{ \begin{array}{l}  \inf_{\xi  > 0} \J( t^{N-1} , [\xi,\xi+D] ) ~,\\
\phantom{\inf_{\xi > 0}} \J(1,[0,D])
\end{array}
\right \} 
\]
\item[Case 4 - \mbox{$N \in (-\infty,0]$}, $\rho < 0$, $D < \frac{\pi}{\sqrt{\delta}}$:]
\[
\I(M^n,g,\mu) \geq \inf_{\xi \in (0,\pi/\sqrt{\delta}-D)} \J\brac{\sin(\sqrt{\delta} t)^{N-1}, [\xi,\xi +D] } .
\]

\end{description}
\end{cor}

The elementary analysis is deferred to Section \ref{sec:model}. By extending our setup in Subsection \ref{subsec:1D} to incorporate the one-dimensional case $n=1$, we think of the above densities as model spaces for the isoperimetric problem under the Curvature-Dimension-Diameter condition $\CDD(\rho,N,D)$ in the above range of parameters, implying the sharpness of our results in this range, at least in the (topological) one-dimensional case $n=1$. 

\medskip

Despite the similarity between Cases 2--4 and the model-spaces which appeared in \cite{EMilmanSharpIsopInqsForCDD} in the case $N \in [n,\infty]$, note that $N-1 < 0$ and that the sign of $\rho$ is reversed, and so the various model densities appearing above are genuinely different, being reciprocals of the ones in \cite{EMilmanSharpIsopInqsForCDD}. Of particular interest is Case 1, which may be called a (positively curved) sphere of (possibly negative) dimension $N \in (-\infty,1)$. This case is of special note since it yields a \emph{single} model density - a property which has thus far been reserved for the case $\rho > 0$ and $N \in [n,\infty]$ (yielding the $N$-sphere when $N < \infty$ and Gaussian measure when $N=\infty$ as the corresponding model spaces).

\begin{rem} Observe that it is not possible to extend Part 2 of Theorem \ref{thm:CDD-II} (at least as is) to the remaining range $N \in (0,1)$ and $D < \infty$. 
For instance, the one-dimensional probability measure $\mu = c_N \sin^{N-1}(t)$ on $(0,\pi)$ with $N \in (0,1)$ satisfies $\text{CDD}(N-1,N,\pi)$, but it's isoperimetric and hence Gromov--L\'evy profiles $\I([0,\pi],\abs{\cdot},\mu)$ and $\GL_{N-1,N,\pi}$ do not coincide with their flat counterpart $\J([0,\pi],\abs{\cdot},\mu) =\J_{N-1,N,\pi}$, since an isoperimetric minimizer of small measure will be a symmetric interval around $\pi/2$, rather than a set of the form $(0,t)$ (whose boundary measure tends to $+\infty$ as $t \rightarrow 0$). Consequently, the analysis of the range $N \in (0,1)$ and $D < \infty$ was excluded from Corollary \ref{cor:model}. While we expect that this range will not have a similar aesthetically pleasing description as the one given in Corollary \ref{cor:model}, note that Part 1 of Theorem \ref{thm:CDD-II} should still yield a meaningful result, but we refrain from developing this here. 
\end{rem}

\subsection{Non-Negative Curvature}

In Section \ref{sec:non-neg}, we investigate the isoperimetric and functional consequences of the $\CD(0,N)$ condition for $N < 0$, when instead of a diameter upper bound, we are given some other information - in the form of a concentration inequality. The results we obtain are appropriately modified analogues of our previous results from \cite{EMilman-RoleOfConvexity,EMilmanGeometricApproachPartI,EMilmanIsoperimetricBoundsOnManifolds,EMilmanGeometricApproachPartII} obtained for $N \in [n,\infty]$, and the results of Bobkov from \cite{BobkovConvexHeavyTailedMeasures} obtained in the Euclidean setting. 
In particular, we show that the function $\I(M,g,\mu)^{\frac{N}{N-1}}$ is weakly-concave, and hence the $N$-dimensional Cheeger isoperimetric constant, which we define as:
\[
 D_{Che,N}(M,g,\mu) := \inf_{v \in (0,1)} \frac{\I(M,g,\mu)(v)}{\min(v,1-v)^{\frac{N-1}{N}}} ,
\]
always attains its infimum at $v = 1/2$. It follows that always $D_{Che,N} > 0$ and that any concentration inequality immediately translates to a positive lower bound on $D_{Che,N}$, and hence upgrades to an $N$-degree polynomial concentration - see Theorem \ref{thm:Isop-Conc-NSpace}. This is used to deduce stability results for $D_{Che,N}$ under various types of measure perturbations (Theorem \ref{thm:stability}), and to assert the equivalence of weak $L^p$-$L^q$ Sobolev inequalities when $\frac{1}{q} \leq \frac{1}{p} + \frac{1}{N}$ (Theorem \ref{thm:main-functional}). In Theorem \ref{thm:main-functional}, we also deduce the following weak Sobolev and Nash-type inequalities on $\CD(0,N)$ weighted manifolds:
\begin{enumerate}
\item For any locally Lipschitz function $f : (M,g) \rightarrow \Real$ with zero median, we have:
\[
\norm{f}_{L^{p,p \frac{N-1}{N}}(\mu)} \leq D_{Che,N}^{-1}  2^{-\frac{1}{N}} p \brac{\frac{q}{p}}^{\frac{1}{q}} \norm{\abs{\nabla f}}_{L^{q,p \frac{N-1}{N}}(\mu)} , 
\]
for any $p,q$ satisfying $\frac{N}{N-1} \leq p \leq -N$ and $\frac{1}{q} = \frac{1}{p} + \frac{1}{N}$. Here $L^{\alpha,r}(\mu)$ denotes the Lorentz quasi-norm (see Definition \ref{def:Lorentz} for our particular choice of normalization). 
\item
For any essentially bounded locally Lipschitz function $f : (M,g,\mu) \rightarrow \Real$ with zero median and $p \geq 1$:
\[
\norm{f}_{L^p(\mu)} \leq D_{Che,N}^{-\frac{N}{N-p}} 2^{-\frac{1}{N-p}}  p^{\frac{N}{N-p}}  \norm{\abs{\nabla f}}_{L^p(\mu)}^{\frac{N}{N-p}}  \norm{f}_{L^\infty(\mu)}^{-\frac{p}{N-p}} ~.
\]
\end{enumerate}

\subsection{Positive Curvature}

In Section \ref{sec:pos}, we consider weighted manifolds satisfying $\CD(\rho,N)$ with $N<1$ and $\rho > 0$. By Case (1) of Corollary \ref{cor:model}, their isoperimetric profile is governed by that of the model density $\cosh^{N-1}(\sqrt{\rho/(1-N)} t)$ on $\Real$. This has some curious consequences: it implies a uniform Poincar\'e inequality for all $N \in (-\infty,1-\eps]$, in contrast to the Lichnerowicz estimate obtained by Kolesnikov--Milman \cite{KolesnikovEMilmanReillyPart1} and Ohta \cite{Ohta-NegativeN}, which explodes as $N < 0$ increases to $0$. Furthermore, it implies a two-level concentration inequality, having tail-decay of the form $\exp(-\min(\rho t^2, \sqrt{\rho} \sqrt{1-N} t))$, which constitutes an interesting intermediate behaviour between a Poincar\'e and log-Sobolev inequality. 
Such two-level concentration naturally appears in Bernstein-type estimates on deviation of sums of independent random-variables having exponential tail decay (see e.g. \cite{LedouxTalagrand-Book,BLM-Book}), but we did not find an explicit connection between the latter setup and the former one. A concrete example of a family of certain harmonic measures on the sphere satisfying $\CD(\rho,N)$ with $N<1$ and $\rho > 0$ is provided in \cite{EMilmanHarmonicMeasures}. 

\medskip

Concluding remarks are made in Section \ref{sec:conclude}. In particular, we comment there regarding additional previously known results and methods for establishing isoperimetric and functional inequalities under a Curvature-Dimension condition. We also comment on our subsequent work \cite{EMilman-GradedCD}, in which we introduce and study the Graded Curvature-Dimension condition. 

\medskip
\noindent \textbf{Acknowledgement.} I warmly thank Will Wylie for informing me about his results on sectional curvature, and the referee for very carefully reading the manuscript.

\section{Generalized Heintze--Karcher Inequality for $N \in (-\infty,1)$} \label{sec:HK}

In this section, we extend the generalized Heintze--Karcher Theorem to the entire range $N \in (-\infty,1) \cup [n,\infty]$. The classical case $N=n$ (constant density) is due to Heintze and Karcher \cite{HeintzeKarcher},\cite[Theorem 4.21]{GHLBookEdition3}. The extension to the range $N \in (n,\infty]$ in the weighted manifold setting is due to Bayle \cite[Appendix E]{BayleThesis} ($N \in (n,\infty)$) and Morgan \cite{MorganManifoldsWithDensity} ($N=\infty$).

Given a $C^2$ hypersurface $S$ in $(M^n,g)$ co-oriented by a unit normal vector field $\nu$, recall that the \emph{mean-curvature} of $S$ at $x$, denoted $H^\nu_S(x)$, is defined as the trace of the second fundamental form $\II^\nu_{S,x}$; it governs the first variation of area under normal deformations. We conform to the following \emph{non-standard} convention for specifying the sign of $\II^\nu_{S,x}$: the second fundamental form of the sphere in Euclidean space with respect to the \emph{outward} normal is \emph{positive} definite (formally: $\II^\nu_{S,x}(u,v) = g(\nabla_u \nu , v)$ for $u,v \in T_x S$). In the weighted manifold setting, the first variation of weighted area is governed by the \emph{generalized mean-curvature} (see \cite{BayleThesis} or the proof of Theorem \ref{thm:HK} below):

\begin{defn*}
The generalized mean-curvature of $S$ at $x \in S$ with respect to the measure $\mu$ and unit normal vector field $\nu$, denoted $H_{S,\mu}^\nu(x)$, is defined as:
\[
H_{S,\mu}^\nu(x) := H_S^\nu(x) + \nu(\log \Psi) (x) ~.
\]
\end{defn*}

\begin{thm}[Generalized Heintze--Karcher] \label{thm:HK}
Let $(M^n,g,\mu)$ satisfy the $\CD(\rho,N)$ condition, $N \in (-\infty,1) \cup [n,\infty]$.
Then it satisfies the $\text{Jac-CD}(\rho,N)$ condition. 

In particular, if $S$ denotes a $C^2$ co-oriented hypersurface in $(M^n,g,\mu)$ with normal unit vector field $\nu$, then for any $r > 0$:
\[ \mu(S_r^+) \leq \int_S \int_0^{r} J_{H^\nu_{S,\mu}(x),\rho,N}(t) dt \; dvol_{S,\mu}(x) ~,
\]
where: \[
S_r^+ := \set{\exp_x(t \nu(x)) ; x \in S , t \in [0,r], t \nu(x) \in \text{Dom}(\exp_x) } ~.
\]
\end{thm}
\begin{proof}
Recall that $F_S(x,t) = \exp_x( t \nu(x))$ denotes the normal map. Fix $x \in S$, set:
\[
 a_{\pm} = a_{\pm}(x) := \pm \sup \set{ t \geq 0 \; ; \;  \pm t \nu(x) \in \text{Dom}(\exp_x)}
,
\]
and denote $\nu_t(x) := \frac{d}{dt} F_S(x,t)$.  
By contracting the Jacobi equation and applying Cauchy--Schwarz, it is classical (see e.g. \cite[p. 72]{ChavelEigenvalues}) that the (unweighted) Jacobian $J_{S,x}(t)$ satisfies the following Riccati-type second-order differential inequality:
\begin{equation} \label{eq:HK1}
- (\log J_{S,x})''(t) - \frac{1}{n-1} ((\log J_{S,x})'(t))^2 \geq Ric_g (\nu_t(x) , \nu_t(x)) ,
\end{equation}
for any $t$ in a sub-interval of  $[a_-,a_+]$ containing the origin where $J_{S,x}$ remains positive (i.e. between the first negative and positive focal points). This already verifies the assertion in the constant density case $N=n$, so we may subsequently assume that $N \notin \set{n,\infty}$; the case of $N=\infty$ follows by an obvious modification of the argument below. 

Now, denoting $J_{W,x}(t) = \Psi(F(x,t)) / \Psi(x)$, we trivially have:
\begin{equation} \label{eq:HK2}
-(\log J_{W,x})''(t) - \frac{1}{N-n} ((\log J_{W,x})'(t))^2 = \brac{ - \nabla_g^2 \log \Psi - \frac{1}{N-n} \nabla_g \Psi \otimes \nabla_g \Psi }(\nu_t(x),\nu_t(x)) .
\end{equation}
Note that the right-hand sides of (\ref{eq:HK1}) and (\ref{eq:HK2}) sum up to precisely $Ric_{g,\mu,N}(\nu_t(x) , \nu_t(x))$, which by the $\CD(\rho,N)$ assumption is at least $\rho$. 
It remains to sum the left-hand sides above, noting that $\log J_{S,\mu,x} = \log J_{S,x} + \log J_{W,x}$. The quadratic non-linearity in the first-order derivatives is handled by an application of Cauchy--Schwarz in the form:
\begin{equation} \label{eq:HK-CS}
\frac{1}{\alpha} A^2 + \frac{1}{\beta} B^2 \geq \frac{1}{\alpha + \beta} (A+B)^2 \;\;\; \forall A,B \in \Real ~,
\end{equation}
valid as soon as $(\alpha,\beta)$ lay in either the set $\set{\alpha, \beta > 0}$ or the set $\set{ \alpha + \beta < 0 \text{ and }\alpha \beta < 0}$, and this is precisely ensured by the assumption that $N \in (-\infty,1) \cup (n,\infty)$. It follows that:
\begin{equation} \label{eq:HK3}
-(\log J_{S,\mu,x})''(t) - \frac{1}{N-1} ((\log J_{S,\mu,x})'(t))^2  \geq \rho ,
\end{equation}
for any $t$ in a sub-interval of  $[a_-,a_+]$ containing the origin where $J_{S,\mu,x}$ remains positive, thereby verifying the $\text{Jac-CD}(\rho,N)$ condition. 

To show the second part of the assertion, note that $J_{S,\mu,x}(0) = J_{S,x}(0) = J_{W,x}(0) = 1$, and that $J_{S,x}'(0) = H^\nu_{S}(x)$ by the classical first variation formula. We thus immediately verify that the first variation of weighted area is indeed given by the generalized mean-curvature:
\[
J_{S,\mu,x}'(0) = J_{S,x}'(0) + J_{W,x}'(0) = H^\nu_S(x) + \nu(\log \Psi)(x) = H^\nu_{S,\mu}(x) .
\]
Setting $\N := N-1$ and rewriting (\ref{eq:HK3}) as:
\[
- \N \frac{(J_{S,\mu,x}^{1/\N})''}{J_{S,\mu,x}^{1/\N}} \geq \rho ,
\]
and setting $h = J_{S,\mu,x}^{1/\N}$, we see that:
\[
\N h''(t) + \rho h(t) \leq 0 \;\; , \;\; h(0) = 1 \;\; , \;\; h'(0) = \frac{H^\nu_{S,\mu}(x)}{\N} ,
\]
for any $t$ in a sub-interval of  $[a_-,a_+]$ containing the origin on which $h$ remains positive.
Comparing with the solution $h_0$ to the equality case in the above ODE, the classical Sturm--Liouville theory (or just calculating the derivative of $h' h_0 - h h_0'$) ensures that:
\begin{equation} \label{eq:N-sign}
\N h(t) \leq \N h_0(t) ,
\end{equation}
for any $t$ in a sub-interval of  $[a_-,a_+]$ containing the origin on which both $h$ and $h_0$ remain positive. Coupled with Remark \ref{rem:J-eq}, we conclude (regardless of the sign of $\N$) that for any $t$ in a sub-interval of  $[a_-,a_+]$ containing the origin on which $J_{S,\mu,x}(t)$ remains positive:
\begin{equation} \label{eq:N-sign2}
J_{S,\mu,x}(t) = h^{\N}(t) \leq ((h_0)_+(t))^{\N} =  J_{H^\nu_{S,\mu}(x),\rho,N}(t) , \end{equation}
and hence for all $t \in [a_-,a_+]$:
\begin{equation} \label{eq:HK4}  (J_{S,\mu,x})_+(t) \leq  J_{H^\nu_{S,\mu}(x),\rho,N}(t) .
\end{equation}

It now remains to use the well-known fact (see e.g. \cite{HeintzeKarcher,ChavelRiemannianGeometry1stEd}) that the normal map $F_S(x,t)$ remains onto $S^+_r$ even when extending the normal ray only until the first focal point (first positive root of $J_{S,x}(t)$ and hence $J_{S,\mu,x}(t)$). Consequently:
\[
\mu(S_r^+) \leq \int_S \int_0^{\min(r,a_+(x))} (J_{S,\mu,x})_+(t) \; dt \; dvol_{S,\mu}(x) ,
\]
and using (\ref{eq:HK4}), the assertion follows. 
\end{proof}

\section{Isoperimetric Inequalities} \label{sec:isop}

\subsection{Extending Setup to One-Dimensional Case} \label{subsec:1D}

We begin with extending our setup to the one-dimensional case. 

\begin{defn*}
The one-dimensional space $(M,|\cdot|,\mu)$ with $M \subset \Real$ is said to satisfy the $\CDD(\rho,N,D)$ condition, $N \in (-\infty,\infty]$, $\rho \in \Real$, $D \in (0,\infty]$, if $\mu$ is a  measure supported on the closure of an open interval $L \subset M$ of length at most $D$, having positive $C^2$-smooth density $\Psi$ on $L$, and satisfying:
\begin{equation} \label{eq:1d-ODE}
-\LogHess_{N-1} \Psi \geq \rho \; \text{ in } L ,
\end{equation}
(with the usual interpretation when $N=1$ or $N=\infty$). When $D=\infty$, the $\CDD(\rho,N,\infty)$ condition is also denoted $\CD(\rho,N)$. 
\end{defn*}

\subsection{Elementary Properties of $J_{\rho,N,D}$}

\begin{lem} \label{lem:monotone} Given $t , \rho\in \Real$ and $N \in (-\infty,\infty]$, the function $\Real \ni H\rightarrow J_{H,\rho,N}(t) \in [0,\infty]$
is continuous and monotone in $H$. Namely, it is constant if $t = 0$ or $N=1$, otherwise $H \mapsto J_{t H,\rho,N}(t)$ is monotone non-decreasing with $\lim_{H \rightarrow -\infty}  J_{t H,\rho,N}(t) = 0$ and $\lim_{H \rightarrow +\infty}  J_{t H,\rho,N}(t) = \infty$.
\end{lem}
\begin{proof} The continuity is immediate and limiting behavior is verified by direct inspection. To show the monotonicity when $N \neq 1$, set $\N = N-1$ and consider $f = k_2' k_1 - k_1' k_2$ where $k_i = \N J_{H_i,\rho,N}^{1/\N}$ with $H_2 > H_1$. Note that $f(0) = H_2 - H_1 > 0$ and $f' \equiv 0$ on any interval $L$ containing the origin such that both $k_i$'s remain positive. Consequently,  $\log (k_2/k_1)$ is monotone increasing on $L$ if $\N > 0$ (decreasing if $\N < 0$) whilst vanishing at the origin, and hence $\log (J_{H_2,\rho,N} / J_{H_1,\rho,N})$ is monotone increasing on $L$ for all $N \neq 1$, concluding the proof. 
\end{proof}

\begin{cor} \label{cor:v0} For all $\rho \in \Real$, $N \in (-\infty,\infty]$ and $D \in (0,\infty]$:
\[
v \in \set{0,1} \;\; \Rightarrow \;\;  \GL_{\rho,N,D}(v) = 0 .
\]
\end{cor}
\begin{proof} When $N =1$, we have by definition $\GL_{0,1,D}(0) = \GL_{0,1,D}(1) = 0$, and the case $\rho \neq 0$ is trivial, so let us assume $N \neq 1$. 
Setting $v=0$ and recalling the relevant definition, we have:
\[
\GL_{\rho,N,D}(0) = \inf_{(a,b) \in \Delta_D , H \in \Real} \frac{1}{\int_{0}^b J_{H,\rho,N}(t) dt}  .
\]
The claim then follows by fixing $(a,b) \in \Delta_D$, letting $H \rightarrow +\infty$ and using Lemma \ref{lem:monotone}. When $v=1$, we instead take the limit as $H \rightarrow -\infty$. 
\end{proof}

\begin{prop} \label{prop:eta} Given $H,\rho \in \Real$, $D \in (0,\infty]$, $(a,b) \in \Delta_D$ and $N \in (-\infty,\infty] \setminus \set{1}$, denote by $\eta_{H}$ the measure on $L = [-a,b]$ having density $J_{H,\rho,N}(t) dt$. Set $L_- = [-a,0]$ and $L_+ = [0,b]$. Then:
\begin{enumerate}
\item
For every choice of $\pm$, the function:
\[
\Real \ni H \mapsto \eta_{\pm H}(L_\pm) \in [0,\infty] ,
\]
is monotone non-decreasing, continuous whenever it is finite, lower semi-continuous on $\Real$ and satisfies $\lim_{H \rightarrow \infty} \eta_{\pm H}(L_\pm) = \infty$. 
\item If $D = \infty$ or $N \in (-\infty,0] \cup (1,\infty]$, then for every choice of $\pm$ the above function is continuous on $\Real$. 
\item If $\eta_H(L) = \infty$ for all $H \in \Real$, then for all $v \in [0,1]$:
\begin{equation} \label{eq:zero}
\inf_{H \in \Real} \max \brac{ \frac{v}{\eta_{H}(L_-)} , \frac{1-v}{\eta_{H}(L_+)} } = 0 .
\end{equation}
\end{enumerate}
\end{prop}
\begin{proof}
\hfill
\begin{enumerate}
\item
Lemma \ref{lem:monotone} implies that $\Real \ni H \mapsto \eta_{\pm H}(L_{\pm}) \in [0,\infty]$ is monotone non-decreasing. Since $H \mapsto J_{H,\rho,N}(t)$ is continuous, it follows that $H \mapsto \eta_{H}(L_{\pm})$ is also continuous as long as it is finite, and by Fatou's Lemma, it must be lower-semi-continuous (i.e. the only possible discontinuity may be a lower semi-continuous jump-discontinuity from a finite to an infinite value). Lastly, since when $\pm t > 0$, $\lim_{H \rightarrow \pm\infty} J_{H}(t) = \infty$ by Lemma \ref{lem:monotone}, another application of Fatou's lemma implies that $\lim_{H \rightarrow \pm \infty} \eta_{H}(L_\pm) = \infty$, concluding the proof of the first assertion. 
\item 
As already mentioned above, the continuity is immediate as long as $\eta_{H}(L_{\pm})$ is finite. This is the case for all $H \in \Real$ when $N \in (1,\infty)$ or when $N = \infty$ and $\{ D < \infty$ or $\rho > 0 \}$; the case $N = D = \infty$ and $\rho \leq 0$ yields a trivial conclusion since $\eta_{H}(L_{\pm}) \equiv \infty$. However, the case $N \in (-\infty,1)$ where the densities may be infinite is surprisingly subtle. The only possible obstruction to continuity is a lower semi-continuous jump discontinuity from finite to infinite $\eta_{H}(L_{\pm})$. This can only happen if at one of the end points $t_0 \in \set{-a,b}$ and some $H_0 \in \Real$ we have $J_{H_0,\rho,N}(t_0) = \infty$, while being integrable at $t_0$ from the right ($t_0 = -a$) or left ($t_0 = b$). This cannot happen if $D = \infty$ (no end points to worry about) nor if $N \leq 0$ (since the trigonometric expressions appearing in the definition of $J_{H,\rho,N}$ intersect the abscissa transversely, so if $N \leq 0$, $J_{H,\rho,N}$ will not be integrable there). The asserted continuity is therefore established. 
\item
Set $H_{\max} := \sup \set{ H \in \Real \; ; \; \eta_{H}(L_{-}) = \infty}$ and $H_{\min} := \inf \set{ H \in \Real \; ; \; \eta_{H}(L_{+}) = \infty}$. Since $\eta_H(L_-) + \eta_H(L_+) = \infty$ for all $H \in \Real$, it follows that necessarily $H_{\min} \leq H_{\max}$. If there exists $H \in \Real$ so that $\eta_{H}(L_{-}) = \eta_{H}(L_+) = \infty$ then (\ref{eq:zero}) immediately follows, so we may assume that $H_{\min} = H_{\max} =: H_0$ and that w.l.o.g. $\eta_{H_0}(L_{-}) = \infty$ and $\eta_{H_0}(L_+) < \infty$. Lower semi-continuity then implies that $\lim_{\eps \rightarrow 0+} \eta_{H_0 + \eps}(L_{-}) = \infty$ while $\eta_{H_0+\eps}(L_{+}) = \infty$ for all $\eps > 0$, thereby verifying (\ref{eq:zero}).  
\end{enumerate}
\end{proof}

\subsection{Proof of Theorem \ref{thm:CDD-II}}

By using the generalized Heintze-Karcher Theorem \ref{thm:HK} in the newly added range $N \in (-\infty,1)$, the first part of Theorem \ref{thm:CDD-II} follows by repeating the proof of \cite[Theorem 1.2]{EMilmanSharpIsopInqsForCDD}; for completeness, we sketch the proof. The second part is more subtle, so we present it here in greater detail. 

\begin{proof}[Proof of Theorem \ref{thm:CDD-II}]
Let $(M^n,g,\mu)$ satisfy the $\CDD(\rho,N, D)$ condition, $n \geq 2$, and denote $\I = \I_{(M,g,\mu)}$. By Corollary \ref{cor:v0}, when $v \in \set{0,1}$ there is nothing to prove. Given $v \in (0,1)$, it was shown in \cite[Section 2]{EMilmanSharpIsopInqsForCDD} that there exists an open isoperimetric minimizer  $A \subset (M,g)$ with $\mu(A) = v$, $\mu^+(A) = \mu^+(M \setminus A) = \I(v)$ and $\mu(\partial A) = 0$. Furthermore, it was shown that there exists a $C^2$ relatively open hypersurface $S \subset \partial A$ (the regular part of the boundary in $int(M)$, the interior of $M$) co-oriented by a unit normal vector field $\nu$, so that:
\begin{enumerate}
\item $vol_{S,\mu}(S) = \I(v)$.
\item
$S$ has constant generalized mean-curvature $H_{S,\mu}^\nu \equiv H$.
\item
For some $(a,b) \in \Delta_D$:
\begin{equation} \label{eq:proof-II}
int(M) \setminus \overline{A} \subset S^+_b ~,~ int(M) \cap A \subset S^-_a .
\end{equation}
Here $S^-_a = (-S)^+_a$, where $-S$ denotes the hypersurface $S$ with reversed co-orientation. 
\end{enumerate}
The proof in \cite{EMilmanSharpIsopInqsForCDD} crucially relied on existence and regularity of isoperimetric minimizers provided by Geometric Measure Theory, as well as on the geodesic convexity of $M$. 

It follows from (\ref{eq:proof-II}) and the generalized Heintze-Karcher Theorem \ref{thm:HK} that:
\begin{align*}
1-v \leq \mu(S^+_b) \leq & \int_S \int_0^{b} J_{H^\nu_{S,\mu}(x),\rho,N}(t) dt \; dvol_{S,\mu}(x) ~, \\
v \leq \mu((-S)^+_a) \leq & \int_S \int_0^{a} J_{H^{-\nu}_{S,\mu}(x),\rho,N}(t) dt \; dvol_{S,\mu}(x) ~.
\end{align*}
Recalling that $S$ has constant generalized mean-curvature so that $H^\nu_{S,\mu} \equiv H$ and $H^{-\nu}_{S,\mu} \equiv -H$, that $vol_{S,\mu}(S) = \I(v)$, and noting that $J_{-H,\rho,N}(t) = J_{H,\rho,N}(-t)$, we conclude that:
\begin{equation} \label{eq:max}
\I(v) \geq \inf_{H \in \Real, (a,b) \in \Delta_D} \max \brac{ \frac{v}{\int_{-a}^0 J_{H,\rho,N}(t) dt} , \frac{1-v}{\int_0^b J_{H,\rho,N}(t) dt}} = \GL_{\rho,N,D}(v) ,
\end{equation}
yielding the first part of the assertion in the case $n \geq 2$. The case $n=1$ is handled by verbatim repeating the proof of
\cite[Corollary 3.2]{EMilmanSharpIsopInqsForCDD}, which remains valid in the entire range $N \in (-\infty,\infty]$ (as in (\ref{eq:N-sign}) and (\ref{eq:N-sign2}), the sign of $N-1$ plays no role in the analysis of the one-dimensional differential inequality). 

Next, given $v \in (0,1)$, we claim that:
\begin{equation} \label{eq:J>GL} \J_{\rho, N, D}(v) \geq \GL_{\rho,N,D}(v). 
\end{equation} For $N =1$, since $\inf_{w \in (0,1)} \max (\frac{v}{w} , \frac{1-v}{1-w}) = 1$, direct inspection verifies that we have equality in (\ref{eq:J>GL}), and so we henceforth exclude this case. Recall that:
\[
\J_{\rho, N, D}(v) := \inf_{(a,b) \in \Delta_D} \inf_{H \in \Real} \set{ \J\brac{ J_{H,\rho,N}, [-a,b] }(v) ;  \int_{-a}^b J_{H,\rho,N}(t) dt < \infty},
\]
with the inner infimum over an empty set of $H$'s interpreted as $0$. 

Assume first that for all $(a,b) \in \Delta_D$, there exists at least one $H \in \Real$ so that $Z = \int_{-a}^b J_{H,\rho,N}(t) dt < \infty$. For all such triplets $a,b,H$, since the one-dimensional space $([-a,b],\abs{\cdot},\frac{1}{Z} J_{H,\rho,N})$ satisfies
$\CDD(\rho,N,D)$ by definition, it follows by the first assertion that:
\[
\J\brac{ J_{H,\rho,N}, [-a,b] }(v) \geq \I\brac{ J_{H,\rho,N}, [-a,b] }(v) \geq \GL_{\rho,N,D}(v) ,
\]
and so taking the infimum over $a,b,H$ as above, (\ref{eq:J>GL}) follows.

On the other hand, if for some $(a,b) \in \Delta_D$ we have $\int_{-a}^b J_{H,\rho,N}(t) dt = \infty$ for all $H \in \Real$, so that $\J_{\rho,N,D}(v) = 0$, we also have by Proposition \ref{prop:eta}:
\[ \GL_{\rho,N,D}(v) \leq \inf_{H \in \Real} \max \brac{\frac{v}{\int_{-a}^0 J_{H,\rho,N}(t) dt} , \frac{1-v}{\int_{0}^b J_{H,\rho,N}(t) dt} } = 0 ,
\] confirming (\ref{eq:J>GL}) as well. 

\smallskip
To show the reverse inequality to (\ref{eq:J>GL}) when $D = \infty$ or $N \in (-\infty,0] \cup (1,\infty]$, we will require the following lemma, which we state very generally for future use; it was implicitly proved in \cite{EMilmanSharpIsopInqsForCDD} under the assumption that all measures below have finite continuous densities.

\begin{lem} \label{lem:model-equiv}
Let $L \subset \Real$ denote an interval containing the origin in its interior, and set $L_{\pm} = L \cap [0,\pm \infty)$ (with $[0,-\infty) := (-\infty,0]$). 
Let $\set{\eta_H}_{H \in \Real}$ denote a family of Borel measures on $L$, so that:
\begin{itemize}
\item The mapping $\Real \ni H \mapsto \eta_{\pm H}(L_{\pm}) \in [0,\infty]$ is continuous, monotone non-decreasing and $\lim_{H \rightarrow +\infty} \eta_{\pm H}(L_{\pm}) = \infty$. 
\item For all $H \in \Real$, $\eta_{H}(\set{0}) = 0$ and $\eta_H^+((-\infty,0]) = 1$.
\end{itemize}
Then for any  $v \in (0,1)$ we have:
\begin{equation} \label{eq:max-lem}
\inf_{H \in \Real} \max \brac{ \frac{v}{\eta_H(L_{-})} , \frac{1-v}{\eta_H(L_+)}} \geq \inf_{H \in \Real, \eta_H(L) < \infty} \J(\eta_H, L)(v) ,
\end{equation}
where the above infimum over an empty set is interpreted as $0$. 
\end{lem}
\begin{proof}
We are given that $\lim_{H \rightarrow \infty} \eta_H(L_+) = \lim_{H \rightarrow -\infty} \eta(L_{-}) = \infty$. As for the limits in the opposite direction, either $\eta_H(L_+) = \infty$ for all $H \in \Real$, in which case the infimum on the right-hand-side of (\ref{eq:max-lem}) is over an empty set and there is nothing to prove, or else by monotonicity $\lim_{H \rightarrow -\infty} \eta_H(L_+) = c_+  \in [0,\infty)$. Similarly, either $\eta_H(L_{-})  = \infty$ for all $H \in \Real$ and there is nothing to prove, or $\lim_{H \rightarrow \infty} \eta_H(L_{-}) = c_-  \in [0,\infty)$. So in the only remaining case requiring proof, the conditions ensure that the first (second) term in the maximum in (\ref{eq:max-lem}) varies continuously and monotonically from $0$ to $\frac{v}{c_{-}}$ ($\frac{1-v}{c_+}$ to $0$) as $H$ varies from $-\infty$ to $\infty$. Consequently, there must exist $H_0 \in \Real$ so that both terms are equal and it is there that the infimum of the left-hand-side in (\ref{eq:max-lem}) is attained:
\[
\frac{\eta_{H_0}(L_{-})}{v} = \frac{\eta_{H_0}(L_+)}{1-v} = \eta_{H_0}(L_{-}) + \eta_{H_0}(L_+) = \eta_{H_0}(L) \in [0,\infty] ~;
\]
If $\eta_{H_0}(L)  = \infty$ then $\eta_{H_0}(L_{-}) = \eta_{H_0}(L_+) = \infty$, and hence by monotonicity necessarily $\eta_{H}(L) = \infty$ for all $H \in \Real$, and there is nothing to prove. If $\eta_{H_0}(L)  = 0$ then the left-hand-side of (\ref{eq:max-lem}) is equal to $\infty$ and there is nothing to prove again. Consequently, we may assume that $\eta_{H_0}(L) \in (0,\infty)$.  

Denoting $\mu = \eta_{H_0} / \eta_{H_0}(L)$, note that $\mu(L_{-}) = v$ and $\mu(L_+) = 1-v$, and that $\mu^+(L_{-}) = \frac{1}{\eta_{H_0}(L)}$. 
Consequently:
\[
\inf_{H \in \Real} \max \brac{ \frac{v}{\eta_H(L_-)} , \frac{1-v}{\eta_H(L_+)}}  = \frac{1}{\eta_{H_0}(L)} = \mu^+(L_-) \geq \J(\eta_{H_0},L)(v),
\]
concluding the proof. 
\end{proof}

Proposition \ref{prop:eta} ensures that the family of measures $\eta_H = J_{H,\rho,N}(t) dt$ on $L = [-a,b]$ with $(a,b) \in \Delta_D$ satisfies all the conditions of the preceding lemma if $D = \infty$ or $N \in (-\infty,0] \cup (1,\infty]$. Consequently, Lemma \ref{lem:model-equiv} yields:
\[
\GL_{\rho,N,D}(v) \geq \inf_{(a,b) \in \Delta_D} \inf_{H \in \Real , \int_{-a}^b J_{H,\rho,N}(t) dt < \infty} \J(J_{H,\rho,N},[-a,b])(v) = \J_{\rho,N,D}(v) ,
\]
where the inner infimum over an empty set is interpreted as $0$. This concludes the proof. 
\end{proof}

\section{New Model Spaces} \label{sec:model}

As in \cite{EMilmanSharpIsopInqsForCDD}, the very general $\CDD(\rho,N,D)$ isoperimetric inequality may be equivalently reformulated in a more explicit manner, according to the different values of the three parameters $\rho$, $N$ and $D$. When $N \in [n,\infty]$, seven cases were described in \cite{EMilmanSharpIsopInqsForCDD}. Here, we add four more cases corresponding to the sub-range of $N \in (-\infty,1)$ and $D \in (0,\infty]$ when Theorem \ref{thm:CDD-II} (2) applies. 

\begin{proof}[Proof of Corollary \ref{cor:model}]

Set $\N := N-1 \in (-\infty,0)$, recall that $\delta := \rho / \N$, and denote if $\rho \neq 0$:
\[
\beta := \frac{H}{\N \sqrt{\abs{\delta}}} .
\]

\noindent
\textbf{Cases 1,2.}
Let $\rho > 0$ and observe that:
\[
J_{H,\rho,N}(t) = \brac{\brac{\cosh(\sqrt{-\delta}t) + \beta \sinh(\sqrt{-\delta}t)}_+}^{\N} =
\begin{cases} \brac{\brac{\frac{\sinh(\alpha + \sqrt{-\delta}t)}{\sinh(\alpha)}}_+}^{\N} & \abs{\beta} > 1 \\
\brac{\frac{\cosh(\alpha + \sqrt{-\delta}t)}{\cosh(\alpha)}}^{\N} & \abs{\beta} < 1 \\
\exp(\sqrt{-\delta} t)^{\N} & \beta = 1 \\
\exp(-\sqrt{-\delta} t)^{\N} & \beta = -1
\end{cases} ~,
\]
where:
\[
\alpha := \begin{cases} \arccoth(\beta) \in \Real \setminus \set{0} & \abs{\beta} > 1 \\ \arctanh(\beta) \in \Real & \abs{\beta} < 1 \end{cases} ~.
\]

When $D = \infty$, the only scenario where $J_{H,\rho,N}$ is integrable on the entire $\Real$ is when $\abs{\beta} < 1$, and consequently Case 1 follows:
\[
\inf_{H \in \Real , \int_{-\infty}^\infty J_{H,\rho,N}(t) dt < \infty } \J(J_{H,\rho,N}(t),\Real) =  \J(\cosh(\sqrt{-\delta} t)^{\N}, \Real ) .
\]

When $D < \infty$, we are required to assume $\N = N-1 \in (-\infty,-1]$, and we easily obtain for all $v \in (0,1)$:

\begin{align*}
& \inf_{(a,b) \in \Delta_D} \inf_{H \in \Real , \int_{-a}^b J_{H,\rho,N}(t) dt < \infty } \J(J_{H,\rho,N}(t),[-a,b])(v) \\
& = \min \left \{ \begin{array}{l}
\inf_{\xi > 0 } \J( \sinh(\sqrt{-\delta} t)^{\N}, [\xi,\xi+D] )(v) ~ , \\
\inf_{\xi \in \Real} \J( \exp(-\sqrt{-\delta} t)^{\N} , [\xi,\xi+D] )(v) ~, \\
\inf_{\xi \in \Real} \J(\cosh(\sqrt{-\delta} t)^{\N},[\xi,\xi+D] )(v) \\
\end{array}
\right \} ~,
\end{align*}
By scale invariance of the exponential function, the second infimum need only be tested at $\xi=0$. This concludes the proof of Case 2. 

\medskip

\noindent
\textbf{Case 3.} Assume now that $\rho=0$, $D<\infty$ and $\N \in (-\infty,-1]$. The claim follows by taking the limit as $\rho \rightarrow 0$ in Case 2, but this requires justification. We prefer to deduce the assertion directly. Indeed, note that:
\[
J_{H,0,N}(t) = \brac{\brac{1 + \frac{H}{\N} t}_+}^{\N} ~.
\]
Observe that when $H=0$ we obtain the uniform density, and so $\J(J_{0,0,N}(t),[-a,b])(v) = \frac{1}{D} = \J(1,[0,D])(v)$ for all $v \in (0,1)$ and $(a,b) \in \Delta_D$. When $H \neq 0$, we may translate by setting $s = t + \frac{\N}{H}$, obtaining for any $v \in (0,1)$:
\begin{align*}
& \inf_{(a,b) \in \Delta_D} \inf_{H \in \Real \setminus \set{0}, \int_{-a}^b J_{H,0,N}(t) dt < \infty} \J(J_{H,0,N}(t),[-a,b])(v) \\
&  = \inf_{a \in [0,D]} \inf_{H \in \Real \setminus \set{0}, \int_{-a}^b J_{H,0,N}(t) dt < \infty} \J\brac{(s_+)^{\N} , \left [\frac{\N}{H}-a , \frac{\N}{H}+D-a \right]}(v) \\
& = \inf_{\xi > 0} \J(s^{\N}_+,[\xi,\xi+D])(v) ~.
\end{align*}
Note that the uniform density in the formulation of the lower bound given in Case 3 was only added for completeness of the description of all model densities, since it may be attained as the limiting case when $\xi \rightarrow \infty$. This concludes the proof of Case 3. 
 
\medskip

\noindent
\textbf{Case 4.} Assume now that $\rho<0$, $D \in (0,\pi / \sqrt{\delta})$ and $\N \in (-\infty,-1]$, and observe that:
\[
J_{H,\rho,N}(t) = \brac{\brac{\cos(\sqrt{\delta}t) + \beta \sin(\sqrt{\delta}t)}_+}^{\N} = \brac{\brac{\frac{\sin(\alpha + \sqrt{\delta}t)}{\sin(\alpha)}}_+}^{\N} ~,
\]
where:
\[
\alpha := \arccot \brac{\beta} \in (0,\pi) ~.
\]
Performing the change of variables $\xi = \alpha / \sqrt{\delta} + t$, it follows that for all $v \in (0,1)$:
\begin{align*}
& \inf_{(a,b) \in \Delta_D} \inf_{H \in \Real , \int_{-a}^b J_{H,\rho,N}(t) dt < \infty} \J(J_{H,\rho,N}(t),[-a,b])(v) \\
& = \inf_{(a,b) \in \Delta_D } \inf_{\alpha \in (0,\pi) , \int_{-a}^b (\sin(\alpha + \sqrt{\delta} t)_+)^{\N} dt < \infty } \J((\sin(\alpha + \sqrt{\delta} t)_+)^{\N},[-a,b])(v) \\
& = \inf_{\xi \in (0,\pi/\sqrt{\delta}-D)} \J((\sin(\sqrt{\delta} t)_+)^{\N},[\xi,\xi+D])(v) ~.
\end{align*}
This concludes the proof. 
\end{proof}

\section{Isoperimetric, Functional and Concentration inequalities under non-negative curvature} \label{sec:non-neg}

\subsection{Concentration} 

Recall that given a metric measure space $(\Omega,d,\mu)$ the concentration profile $\K = \K(\Omega,d,\mu)$ is defined as the pointwise maximal function $\K:\Real_+ \rightarrow [0,1/2]$ such that $1 - \mu(A^d_r) \leq \K(r)$ for all Borel sets $A \subset \Omega$ with $\mu(A) \geq 1/2$. Concentration inequalities measure how tightly the measure $\mu$ is concentrated around sets having measure $1/2$ as a function of the distance $r$ away from these sets, by providing a non-trivial upper bound on $\K$. We refer to \cite{Ledoux-Book} for a wider exposition on these and related topics.

We also recall that an isoperimetric inequality always implies a concentration inequality, simply by integrating along the differential inequality it defines on the concentration profile. Consequently, the following relation between the isoperimetric and concentration profiles is easily verified (e.g. \cite[Section 2]{BobkovHoudreMemoirs}):
\begin{equation} \label{eq:isop-conc}
\K^{-1}(v) \leq \int_v^{1/2} \frac{ds}{\I(s)} .
\end{equation}

\subsection{$N$-dimensional Cheeger constant}

Let $N \in (-\infty,\infty] \setminus \set{0}$; when $N=\infty$, we interpret everything below in the limiting sense. We now define the $N$-dimensional Cheeger constant $D_{Che,N} = D_{Che,N}(\Omega,d,\mu)$ of a metric-measure space $(\Omega,d,\mu)$ having isoperimetric profile $\I = \I(\Omega,d,\mu)$ as:
\begin{equation} \label{eq:def-DCheN}
D_{Che,N} := \inf_{v \in [0,1]} \frac{\I(v)}{\min(v,1-v)^{\frac{N-1}{N}}} .
\end{equation}
When $D_{Che,N} > 0$ we will say that $(\Omega,d,\mu)$ satisfies a $N$-dimensional Cheeger isoperimetric inequality. Plugging $\I(v) \geq D_{Che,N} v^{\frac{N-1}{N}}$ for $v \in [0,1/2]$ into (\ref{eq:isop-conc}), one immediately verifies that an $N$-dimensional Cheeger isoperimetric inequality always implies the following concentration estimate:
\begin{equation} \label{eq:N-conc}
\K(r) \leq \brac{\brac{\frac{1}{2}}^{\frac{1}{N}} - \frac{r D_{Che,N}}{N}}_+^N . 
\end{equation}
In particular, when $D_{Che,N} > 0$ for $N > 0$, we see that the support of $\mu$ must be bounded in $(\Omega,d)$: its diameter is at most $\K^{-1}(0) \leq N (1/2)^{\frac{1}{N}} / D_{Che,N}$. However, there is no such restriction when $N<0$, in which case the estimate (\ref{eq:N-conc}) entails a polynomial-type concentration of degree $N$. 

\medskip

The next theorem and corollaries justify the above terminology:

\begin{thm} \label{thm:Isop-Conc-NSpace}
Assume that $(M^n,g,\mu)$ satisfies the $\CD(0,N)$ condition, $N \in (-\infty,0) \cup [n,\infty]$, and set $\I = \I(M,g,\mu)$, $D_{Che,N} = D_{Che,N}(M,g,\mu)$ and $\K = \K(M,g,\mu)$. Then the following properties hold:
\begin{enumerate}
\item Weak concavity of $\I^{\frac{N}{N-1}}$: $(0,1) \ni v \mapsto \I^{\frac{N}{N-1}}(v) / v$ is non-increasing. 
\item $D_{Che,N} = 2^{\frac{N-1}{N}} \I(1/2)$. 
\item $D_{Che,N} = \sup_{r > 0} 2^{\frac{N-1}{N}} \frac{\frac{1}{2} -\K(r)}{r}$. 
\end{enumerate}
\end{thm}

\begin{rem}
Observe that Part (3) implies that any weak concentration (i.e. $\K(r_0) < 1/2$ for some $r_0 > 0$) on a $\CD(0,N)$ weighted manifold gets automatically upgraded to an $N$-degree polynomial concentration (\ref{eq:N-conc}). 
\end{rem}

\begin{rem}
When $N \geq n$, it is actually known (see \cite{SternbergZumbrun,Kuwert,BayleThesis}) that $\CD(0,N)$ implies the concavity of $\I^{\frac{N}{N-1}}$ on $[0,1]$, which is easily seen to imply the weak concavity asserted in Part (1) above. However, the proof of the concavity involves taking the second normal variation of an isoperimetric minimizer, whereas our proof below of the weak concavity (which is a repetition of our proof in \cite{EMilmanGeometricApproachPartI} for the case $N = \infty$) only requires the easier first variation. Since the weak concavity is enough for all applications we are aware of, we leave the task of verifying the concavity of $\I^{\frac{N}{N-1}}$ when $N < 0$ to the interested reader. Note that in the one-dimensional case $n=1$, this was indeed verified by Bobkov \cite[Lemma 2.2]{BobkovConvexHeavyTailedMeasures}.
\end{rem}

\begin{proof}[Proof of Theorem \ref{thm:Isop-Conc-NSpace}]

Let $A$ denote an isoperimetric minimizer of measure $v \in (0,1)$. We proceed with the notation used in the proof of Theorem \ref{thm:CDD-II}, and recall that the regular part $S$ of the boundary of $A$ in the interior of $M$ has constant generalized mean-curvature $H^\nu_{S,\mu} \equiv H_\mu(A)$, where $S$ is co-oriented by the outer unit normal vector field $\nu$. 

\begin{enumerate}
\item
It is well-known (see e.g. \cite{BavardPansu},\cite{MorganJohnson},\cite[Lemma 3.4.12]{BayleThesis}), that $H_\mu(A)$ is essentially the derivative of $\I$ at $v$, or more precisely:
\begin{equation} \label{eq:H-is-derivative}
\limsup_{\eps\rightarrow 0+} \frac{\I(v+\eps) - \I(v)}{\eps} \leq H_\mu(A) \leq
\liminf_{\eps\rightarrow 0-} \frac{\I(v+\eps) - \I(v)}{\eps} ~.
\end{equation}
Recall also that by (\ref{eq:proof-II}), $int(M) \cap A \subset (-S)^+_\infty$, where $-S$ denotes the hypersurface $S$ with reversed orientation (and hence mean-curvature). Applying the generalized Heintze--Karcher Theorem \ref{thm:HK} to $-S$ (as $N$ is in the required range for the theorem to apply), we obtain:
\[ v = \mu(int(M) \cap A) \leq \mu((-S)^+_D) \leq  \int_S \int_0^{\infty} J_{H^{-\nu}_{S,\mu}(x),0,N}(t) dt \; dvol_{S,\mu}(x) .
\]
But since $H^{-\nu}_{S,\mu} \equiv -H_\mu(A)$ and $vol_{S,\mu}(S) = \mu^+(A) = \I(v)$, it follows that:
\[
\frac{v}{\I(v)} \leq \int_0^\infty J_{-H_\mu(A),0,N}(t) dt = \int_0^\infty \brac{\brac{1 - \frac{H_\mu(A) t}{N-1}}_+}^{N-1} dt 
\]
(where the latter integrand is interpreted as $\exp(- H_\mu(A) t)$ when $N = \infty$). 
When $H_\mu(A) > 0$ and $N \in (-\infty,0) \cup (1,\infty]$, observe that the latter integral is equal to $\frac{N-1}{N} \frac{1}{H_\mu(A)}$ (note that the integral diverges if $N \in [0,1)$, and hence this range has been excluded in the formulation of the theorem). Consequently:
\[
\frac{v}{\I(v)} \leq \frac{N-1}{N} \frac{1}{\max(0,H_\mu(A))} .
\] 

Now regardless of whether $H_\mu(A) > 0$, we conclude using (\ref{eq:H-is-derivative}) that:
\begin{equation} \label{eq:weak-concave}
 \limsup_{\eps\rightarrow 0+} \frac{\I(v+\eps) - \I(v)}{\eps} \leq H_\mu(A) \leq \frac{N-1}{N}\frac{\I(v)}{v} ~.
\end{equation}
Denoting $f(v) = \I(v)/v^{\frac{N-1}{N}}$, it is straightforward to verify that this is equivalent to:
\begin{equation} \label{eq:f-decreases}
\limsup_{\eps\rightarrow 0+} \frac{f(v+\eps) - f(v)}{\eps} \leq 0 .
\end{equation}
It was shown in \cite[Section 6]{EMilman-RoleOfConvexity} that as soon as $\mu$ has density which is locally bounded from above, then $\I$ is continuous on $(0,1)$. Consequently, the function $f$ is also continuous there, and it is immediate to verify that (\ref{eq:f-decreases}) entails that $f$ is non-increasing. We thus confirm that $\I^{\frac{N}{N-1}}(v) / v$ is non-increasing, thereby concluding the proof of the first assertion. 
 
\item
Since $\I(v)/v^{\frac{N-1}{N}}$ is non-increasing by the first part, it follows that the infimum (\ref{eq:def-DCheN}) in the definition of $D_{Che,N}$ is attained at $v=1/2$, yielding the second assertion. 

\item
Now consider an isoperimetric minimizer $A$ of measure $v = 1/2$. Since $\mu(A) = \mu(M \setminus A) = 1/2$ and $\mu^+(A) = \mu^+(M \setminus A) = \I(v)$, by exchanging $A$ with $M\setminus A$ if necessary, we may assume that the constant generalized mean-curvature $H_\mu(A)$ of the regular part of the boundary $S$ in the interior of $M$ (with respect to the outer unit-normal field $\nu$)  satisfies $H_\mu(A) \leq 0$. Given $r > 0$, we apply the generalized Heintze--Karcher theorem as above to bound $\mu(S^+_r)$ from above. The bound from below is provided by the concentration inequality, yielding together:
\[ \frac{1}{2} - \K(r) \leq \mu(S^+_r) \leq \I(1/2) \int_0^r J_{H_\mu(A), 0,N}(t) dt . 
\]
But since $H_\mu(A) \leq 0$ the integrand is bounded above by $1$, and therefore:
\[
D_{Che,N} = 2^{\frac{N-1}{N}} \I(1/2) \geq \sup_{r > 0} 2^{\frac{N-1}{N}} \frac{\frac{1}{2} -\K(r)}{r} . 
\]
The other direction follows by taking the limit $r \rightarrow 0$ in (\ref{eq:N-conc}), thereby concluding the proof of the third assertion. 
 \end{enumerate}
\end{proof}

As a standard corollary, we have:
\begin{cor} \label{cor:Isop-Conc-On-NSpace}
Assume that $(M^n,g,\mu)$ satisfies the $\CD(0,N)$ condition, $N \in (-\infty,0) \cup [n,\infty]$. Then $D_{Che,N} = D_{Che,N}(M,g,\mu) > 0$. Specifically, let $R > 0$ be such that $\mu(B(x_0,R)) \geq 3/4$ for some $x_0 \in M$, where $B(x_0,R)$ denotes the geodesic-ball of radius $R$ around $x_0$. Then $D_{Che,N} \geq 2^{\frac{N-1}{N}}/(8 R)$. In particular, $(M,g,\mu)$ always satisfies an $N$-degree polynomial concentration (\ref{eq:N-conc}). 
\end{cor}

\begin{rem}
Note that we cannot expect in general to obtain better concentration. When $N < 0$, this is witnessed by the $\CD(0,N)$ space $([0,\infty),\abs{\cdot},-N (1+ t)^{N-1} dt)$. The case $N \in [n,\infty]$ has already been treated in the literature, see e.g. \cite{EMilmanGeometricApproachPartI}. 
\end{rem}

\begin{rem}
In the Euclidean setting $(\Real^n,\abs{\cdot},\mu)$, where $\CD(0,N)$ spaces precisely coincide with Borell's class of $\frac{1}{N}$-concave measures (more precisely, with absolutely continuous measures having $C^2$-smooth density in this class), Corollary \ref{cor:Isop-Conc-On-NSpace} was already proved by Bobkov \cite{BobkovConvexHeavyTailedMeasures}, who extended a previous estimate of Kannan--Lov\'asz--Simonovits \cite{KLS} in the case $N=\infty$. 
\end{rem}

\begin{proof}[Proof of Corollary \ref{cor:Isop-Conc-On-NSpace}]
Since $M$ is connected, there always exists an $R>0$ so that $\mu(B(x_0,R)) \geq 3/4$ for any $x_0 \in M$. Let $A$ be any Borel subset of $M$ with $\mu(A) \geq 1/2$. Then $A$ and $B(x_0,R)$ must intersect, and hence $A_{2R} \supset B(x_0,R)$. Consequently, $\mu(A_{2R}) \geq 3/4$ and therefore $\K(2R) \leq 1/4$. The result now immediately follows from Theorem \ref{thm:Isop-Conc-NSpace}. 
\end{proof}

\subsection{Functional Inequalities - Preliminaries}

Using nowadays standard methods (see e.g. \cite{BobkovHoudreMemoirs,BobkovConvexHeavyTailedMeasures} and the references therein), we may rewrite the $N$-dimensional Cheeger inequality as a (weak) Sobolev-type inequality. For the reader's convenience, we sketch the proofs. We begin with some definitions:

\begin{defn}
Given a locally Lipschitz function $f : (M,g) \rightarrow \Real$, its local Lipschitz constant is defined as the following (Borel measurable) function on $(M,g)$ (cf. \cite{BobkovHoudreMemoirs}):
\[
\abs{\nabla f}(x) := \limsup_{y \rightarrow x} \frac{\abs{f(y) - f(x)}}{d(y,x)} ~,
\]
where as usual, $d$ is the induced geodesic distance on $(M,g)$. When $f$ is $C^1$ smooth, the local Lipschitz constant clearly coincides with the Riemannian length of the gradient. 
\end{defn}

\begin{defn}
A median $\med_\mu f$ of a Borel measurable function $f : (M,g,\mu) \rightarrow \Real$ is defined as a median of the push-forward of $\mu$ by $f$, namely, a value $M \in \Real$ so that $\mu(\set{ f \geq M}) \geq 1/2$ and $\mu(\set{ f \leq M}) \geq 1/2$. 
\end{defn}

\begin{defn} \label{def:Lorentz}
Given $\alpha \in (0,\infty)$, $r \in (0,\infty]$ and a measurable real-valued function $f$ on a $\sigma$-finite measure space $(\Omega,\mu)$, recall that the Lorentz quasi-norms are defined as:
\[
\norm{f}_{L^{\alpha,r}(\mu)} := \brac{r \int_0^\infty t^{r} \mu(\set{\abs{f} \geq t})^{r/\alpha} \frac{dt}{t}}^{1/r} ,
\]
with the usual interpretation when $r = \infty$:
\[
\norm{f}_{L^{\alpha,\infty}(\mu)} := \sup_{t > 0} \; t \cdot \mu(\set{\abs{f} \geq t})^{1/\alpha} ~.
\]
Note that our normalization may differ by some numerical constants from other variants used in the literature (our convention coincides with that of \cite{GalingInequalitiesBook}, which ensures that $\norm{1_A}_{L^{\alpha,r}(\mu)} = \mu(A)^{1/\alpha}$ for all $r > 0$, but differs from that of \cite{GrafakosClassicalFourierAnalysis2ndEd}). 
Note that $\norm{f}_{L^{\alpha,\alpha}(\mu)} = \norm{f}_{L^{\alpha}(\mu)}$, and we use the standard convention that $\norm{f}_{L^{\infty,\infty}(\mu)} = \norm{f}_{L^\infty(\mu)}$.
\end{defn}
\begin{rem}
It is well-known (e.g. \cite[Section 1.4]{GrafakosClassicalFourierAnalysis2ndEd}) that these are actual norms when  $1 \leq r \leq \alpha < \infty$, but in general, they are only quasi-norms: 
\[
\norm{f+g}_{L^{\alpha,r}(\mu)} \leq 2^{1/\alpha} \max(1,2^{(1-r)/r}) \brac{\norm{f}_{L^{\alpha,r}(\mu)} + \norm{g}_{L^{\alpha,r}(\mu)}} .
\]
Our normalization is particularly useful since (e.g. \cite[Theorem 10.4.2]{GalingInequalitiesBook}):
\[
0 <  r_1 \leq r_2 \leq \infty \;\;\; \Rightarrow \;\;\; \norm{f}_{L^{\alpha,r_2}(\mu)} \leq  \norm{f}_{L^{\alpha,r_1}(\mu)} ,
\]
and when $\mu$ is a probability measure, we trivially have:
\[
0 < \alpha_1 \leq \alpha_2 \leq \infty \;\;\; \Rightarrow \;\;\; \norm{f}_{L^{\alpha_1,r}(\mu)} \leq \norm{f}_{L^{\alpha_2,r}(\mu)} .
\]
In particular, the $L^\alpha$ quasi-norm is weaker (stronger) than the $L^{\alpha,1}$ one if $\alpha > 1$ ($\alpha < 1$), and both are stronger than the weak $L^{\alpha,\infty}$ quasi-norm, i.e.:  \[
\alpha \in (0,1) \;\;\; \Rightarrow \;\;\; \norm{f}_{L^{\alpha}(\mu)} \geq \norm{f}_{L^{\alpha,1}(\mu)} \geq \norm{f}_{L^{\alpha,\infty}(\mu)} ,
\]
\begin{equation} \label{eq:LorentzBig}
\alpha \geq 1 \;\;\; \Rightarrow \;\;\; \norm{f}_{L^{\alpha,1}(\mu)} \geq  \norm{f}_{L^{\alpha}(\mu)} \geq \norm{f}_{L^{\alpha,\infty}(\mu)} .
\end{equation}
\end{rem}

The following proposition is essentially known to experts, starting from the groundbreaking work of  Federer-Fleming \cite{FedererFleming} and Maz'ya \cite{MazyaSobolevImbedding} in connection to the optimal constant in Gagliardo's inequality on $\Real^n$. When $\mu$ is a probability measure, it reads as follows:

\begin{prop} \label{prop:func-equiv}
Let $(M^n,g,\mu)$ denote a weighted Riemannian manifold and $\alpha > 0$. 
The following statements regarding the existence of a constant $D> 0$ are equivalent:
\begin{enumerate}
\item
For every Borel set $A \subset (M,g)$:
\[
\mu^+(A) \geq D \min(\mu(A),1-\mu(A))^{1/\alpha} ~.
\]
\item
For any non-negative locally Lipschitz function $f : (M,g) \rightarrow \Real_+$ with $\med_\mu(f) = 0$, we have:
\[
\int \abs{\nabla f} d\mu \geq  D  \norm{f}_{L^{\alpha,1}(\mu)} ~.
\]
\item
For any non-negative locally Lipschitz function $f : (M,g) \rightarrow \Real_+$ with $\med_\mu(f) = 0$, we have:
\[
\int \abs{\nabla f} d\mu \geq D \norm{f}_{L^{\alpha,\infty}(\mu)} ~.
\]
\end{enumerate}
When $\alpha \in (0,1)$, statements (1),(2) and (3) are equivalent to the following Nash-type inequality:
\begin{enumerate}
\setcounter{enumi}{3}
\item For any essentially bounded non-negative locally Lipschitz function $f : (M,g,\mu) \rightarrow \Real_+$ with $\med_\mu(f) = 0$, we have:
\[
\norm{f}_{L^1(\mu)} \leq D^{-\alpha}  (\int \abs{\nabla f} d\mu)^{\alpha}  \norm{f}_{L^\infty(\mu)}^{1-\alpha} ~.
\]
\end{enumerate}
When $\alpha \geq 1$, statements (1),(2) and (3) are equivalent to following Gagliardo inequality:
\begin{enumerate}
\setcounter{enumi}{4}
\item For any non-negative locally Lipschitz function $f : (M,g,\mu) \rightarrow \Real_+$ with $\med_\mu(f) = 0$, we have:
\[
\int \abs{\nabla f} d\mu \geq D \norm{f}_{L^{\alpha}(\mu)} ~.
\]
\end{enumerate}
When $\alpha \geq 1$, the assumption on non-negativity of $f$ may be dropped. When $\alpha \in (0,1)$, it may be dropped when passing from one of the above statements to the other, at the expense of multiplying the constant $D$ in the conclusion by an additional $2^{1-\frac{1}{\alpha}}$ factor.
\end{prop}
\begin{proof}[Sketch of Proof]
Clearly $\norm{f}_{L^{\alpha,1}(\mu)} \geq \norm{f}_{L^{\alpha,\infty}(\mu)}$, so statement (2) implies (3). To see that statement (3) implies (1), let $A\subset (M,g)$ denote a Borel set  and $\mu^{+}(A) < \infty$ (otherwise there is nothing to prove). If $\mu(A) \geq 1/2$, set $f_\eps(x) = g_\eps(x) :=\min(\inf_{y \in A} d(x,y) / \eps,1)$, and otherwise set $f_\eps(x) = 1-g_\eps(x)$; letting $\eps \rightarrow 0$, one easily recovers (1) (see \cite[Section 3]{BobkovHoudreMemoirs} for more details). Next, (1) implies (2) by the generalized co-area inequality of Bobkov--Houdr\'e \cite{BobkovHoudreMemoirs}. Indeed, for a non-negative locally-Lipschitz function $f$:
\[
\int \abs{\nabla f} d\mu  \geq \int_{0}^\infty \mu^+(\set{ f > t}) dt  \geq \int_{0}^\infty \mu(\set{f > t})^{1/\alpha} dt  = \norm{f}_{L^{\alpha,1}(\mu)} .
\]
If $f$ is not assumed non-negative, we apply the above argument to $f^+ = \max(f,0)$ and $f^{-} = \max(-f,0)$, and sum the resulting two inequalities, using:
\[
\int\abs{\nabla f} d\mu = \int_{\set{f > 0}} \abs{\nabla f} d\mu + \int_{\set{f < 0}} \abs{\nabla f} d\mu = \int\abs{\nabla f^+} d\mu + \int\abs{\nabla f^-} d\mu .
\]
It remains to note that:
\[
\norm{f^+}_{L^{\alpha,1}(\mu)} + \norm{f^-}_{L^{\alpha,1}(\mu)} \geq 2^{1-\max(\frac{1}{\alpha},1)} \norm{f}_{L^{\alpha,1}(\mu)} .
\]

When $\alpha \in (0,1)$, (2) implies (4) after an application of H\"{o}lder's inequality:
\[
\norm{f}_{L^1(\mu)} = \int_0^{\norm{f}_{L^\infty(\mu)}} \mu(\set{\abs{f} \geq t}) dt \leq \norm{f}_{L^{\alpha,1}(\mu)}^{\alpha} \norm{f}_{L^\infty(\mu)}^{1-\alpha} ~.
\]
Conversely, (4) implies (1) by applying it to $f_\eps$ as described above.

When $\alpha \geq 1$, (2) implies (5) which implies (3) by noting (\ref{eq:LorentzBig}). 
\end{proof}

\begin{rem}
A careful inspection of the proof reveals that if we replace the $L^\infty(\mu)$ norm in (4) by the essential oscillation of $f$ ($\operatorname{ess}\sup_\mu (f) - \operatorname{ess}\inf_\mu(f)$), then there is no need to incur an additional $2^{1-\frac{1}{\alpha}}$ factor when handling arbitrarily signed functions $f$ and $\alpha \in (0,1)$. We refer to \cite[Section 8]{BobkovConvexHeavyTailedMeasures} for more details, and refrain from further pursuing this minor point. 
\end{rem}

The above isoperimetric-type inequalities always imply weaker Sobolev and Nash inequalities. The traditional case $\alpha \geq 1$ is well-known (see e.g. \cite{EMilman-RoleOfConvexity}), so we concentrate on the case $\alpha \in (0,1)$.

\begin{prop} \label{prop:weak-Sobolev}
Let $(M^n,g,\mu)$ denote a weighted Riemannian manifold and $\alpha \in (0,1)$. 
The following isoperimetric-type inequality, asserting the existence of $C >0$ so that for any (non-negative) locally-Lipschitz function $f : (M,g) \rightarrow \Real$ with $\med_\mu f = 0$:
\begin{equation} \label{eq:assump-weak-Sobolev}
\norm{f}_{L^{\alpha,1}(\mu)} \leq C \norm{ \abs{\nabla f} }_{L^1(\mu)}  ,
\end{equation}
implies the following weak Sobolev inequality for any (non-negative) locally-Lipschitz function $g  : (M,g) \rightarrow \Real$ with $\med_\mu g = 0$:
\begin{equation} \label{eq:concl-weak-Sobolev}
\norm{g}_{L^{p,p/\alpha}(\mu)} \leq C p \brac{\frac{q}{p}}^{\frac{1}{q}} \norm{\abs{\nabla g}}_{L^{q,p/\alpha}(\mu)} ,
\end{equation}
for any $p,q$ satisfying $\alpha \leq p \leq \frac{\alpha}{1-\alpha}$, $1 \leq q \leq \infty$ and $\frac{1}{p} - \frac{1}{q} = \frac{1}{\alpha} - 1$. 
\end{prop}
\begin{rem}
Since $1 \leq p/\alpha \leq q$, note that the $L^{q,p/\alpha}(\mu)$ quasi-norm appearing on the right-hand side of (\ref{eq:concl-weak-Sobolev}) is in fact a genuine norm.
\end{rem}
\begin{proof}
Let $g$ denote a (non-negative) locally-Lipschitz function with $\med_\mu g = 0$. Set $f = \sign(g) \abs{g}^{\beta}$ with $\beta = p/\alpha \geq 1$ and note that $\med_\mu f = 0$. Applying (\ref{eq:assump-weak-Sobolev}), observe that:
\[
\snorm{\abs{g}^\beta}_{L^{\alpha,1}(\mu)} = \norm{g}^\beta_{L^{\beta \alpha, \beta}(\mu)} ,
\]
and that by H\"{o}lder's inequality for Lorentz spaces (e.g. \cite[p. 54]{GrafakosClassicalFourierAnalysis2ndEd}):
\[
\int \abs{\nabla f} d\mu = \beta \int \abs{g}^{\beta-1} \abs{\nabla g} d\mu \leq \beta \brac{\frac{q^*}{\beta^*}}^{\frac{1}{q^*}}  \brac{\frac{q}{\beta}}^{\frac{1}{q}} \snorm{\abs{g}^{\beta-1}}_{L^{q^*,\beta^*}(\mu)} \norm{\abs{\nabla g}}_{L^{q,\beta}(\mu)} ,
\]
where $\gamma^*$ denotes the conjugate H\"{o}lder exponent of $\gamma$. Noting that $(\beta-1) q^* = p$ and $(\beta-1) \beta^* = \beta$, we obtain:
\[
\norm{g}^\beta_{L^{p, \beta}(\mu)} \leq C \int \abs{\nabla f} d\mu \leq C p^{\frac{1}{q^*}} q^{\frac{1}{q}} \norm{g}^{\beta-1}_{L^{p,\beta}(\mu)} \norm{\abs{\nabla g}}_{L^{q,\beta}} ,
\]
and the asserted inequality readily follows. 
\end{proof}

\begin{prop} \label{prop:weak-Nash}
Let $(M^n,g,\mu)$ denote a weighted Riemannian manifold and $\alpha \in (0,1)$. 
The following isoperimetric-type inequality, asserting the existence of $C >0$ so that for any (non-negative) locally-Lipschitz function $f : (M,g) \rightarrow \Real$ with $\med_\mu f = 0$:
\begin{equation} \label{eq:assump-weak-Nash}
\norm{f}_{L^1(\mu)} \leq C \norm{\abs{\nabla f}}_{L^1(\mu)}^{\alpha}  \norm{f}_{L^\infty(\mu)}^{1-\alpha} ~.
\end{equation}
implies the following weak Nash inequality for any (non-negative) locally-Lipschitz function $g  : (M,g) \rightarrow \Real$ with $\med_\mu g = 0$:
\[
\norm{g}_{L^{p}(\mu)} \leq C^{\frac{\beta}{\alpha}} p^\beta \norm{\abs{\nabla g}}_{L^p(\mu)}^{\beta}  \norm{g}_{L^\infty(\mu)}^{1-\beta} ~,
\]
for any $p\geq 1$ and $\beta = \frac{\alpha}{\alpha + p(1-\alpha)}$. 
\end{prop}
\begin{proof}
Let $g$ denote a (non-negative) locally-Lipschitz function with $\med_\mu g = 0$. Set $f = \sign(g) \abs{g}^p$ and note as usual that $\med_\mu f = 0$. Applying (\ref{eq:assump-weak-Nash}) followed by H\"{o}lder's inequality, we obtain:
\[
\int \abs{g}^p d\mu \leq C \brac{ p \int \abs{g}^{p-1} \abs{\nabla g} d\mu }^{\alpha} \norm{\abs{g}^p}_{L^\infty(\mu)}^{1-\alpha} \leq
C p^\alpha \brac{\int \abs{g}^p d\mu}^{\alpha(1-\frac{1}{p})} \norm{\abs{\nabla g}}^\alpha_{L^p(\mu)} \norm{g}_{L^\infty(\mu)}^{p(1-\alpha)} ,
\]
and the assertion follows. 
\end{proof}

\begin{rem}
The assumption that $\med_\mu(h) = 0$ in the above propositions may easily be replaced by other standard normalizations, such as $\int h d\mu = 0$ (assuming that $h$ is integrable), leading to an additional factor depending on $\alpha$ in the resulting inequalities. See e.g. the proof of \cite[Lemma 2.1]{EMilman-RoleOfConvexity}.
\end{rem}

\subsection{Functional Inequalities for $N < 0$}

Using Propositions \ref{prop:func-equiv}, \ref{prop:weak-Sobolev} and \ref{prop:weak-Nash}, we may reformulate the results of the previous subsection in a functional language. Since the case $N \geq n$ has already been treated in the literature (e.g. \cite{EMilman-RoleOfConvexity}), we  restrict out attention to the case of negative generalized dimension. The first assertion was obtained by Bobkov in \cite[Section 8]{BobkovConvexHeavyTailedMeasures} in the Euclidean setting.  

\begin{thm} \label{thm:main-functional}
Assume that $(M^n,g,\mu)$ satisfies the $\CD(0,N)$ condition, $N < 0$. Recall that $D_{Che,N} = D_{Che,N}(M,g,\mu) > 0$ by Corollary \ref{cor:Isop-Conc-On-NSpace}. Then $(M,g,\mu)$ satisfies the following:
\begin{enumerate}
\item 
Nash-type inequalities: for any essentially bounded locally Lipschitz function $f : (M,g,\mu) \rightarrow \Real$ with $\med_\mu(f) = 0$ and $p \geq 1$:
\[
\norm{f}_{L^p(\mu)} \leq D_{Che,N}^{-\frac{N}{N-p}} 2^{-\frac{1}{N-p}}  p^{\frac{N}{N-p}}  \norm{\abs{\nabla f}}_{L^p(\mu)}^{\frac{N}{N-p}}  \norm{f}_{L^\infty(\mu)}^{-\frac{p}{N-p}} ~.
\]
\item Weak Sobolev inequalities: for any locally Lipschitz function $f : (M,g) \rightarrow \Real$ with $\med_\mu(f) = 0$, we have:
\begin{equation} \label{eq:best-weak-Sobolev}
\norm{f}_{L^{p,p \frac{N-1}{N}}(\mu)} \leq C_{p,q} \norm{\abs{\nabla f}}_{L^{q,p \frac{N-1}{N}}(\mu)} , 
\end{equation}
for any $p,q$ satisfying $\frac{N}{N-1} \leq p \leq -N$ and $\frac{1}{q} = \frac{1}{p} + \frac{1}{N}$, with:
\begin{equation} \label{eq:Cpq}
C_{p,q} \leq D_{Che,N}^{-1}  2^{-\frac{1}{N}} p \brac{\frac{q}{p}}^{\frac{1}{q}} .
\end{equation}
The case $p=\frac{N}{N-1}$ and $q=1$ is a weak Gagliardo inequality.
\item All of the above weak-Sobolev inequalities are equivalent in the following sense: if $C_{p,q}$ denotes the best constant so that (\ref{eq:best-weak-Sobolev}) is satisfies for any locally Lipschitz function $f : (M,g) \rightarrow \Real$ with $\med_\mu(f) = 0$, then:
\[
C_{p_2,q_2} \leq 2^{1 + \frac{2}{p_1}} p_2 \brac{1 + \frac{p_2}{N}}^{-(\frac{1}{p_2} + \frac{1}{N})} C_{p_1,q_1} \;\; , \;\; \text{for all } \;  \frac{1}{q_i} \leq \frac{1}{p_i} + \frac{1}{N} \; , \; i=1,2 . 
\]
\end{enumerate}
\end{thm}

\begin{proof}
By Proposition \ref{prop:func-equiv} with $\alpha = \frac{N}{N-1}$ and the definition of $D_{Che,N}$, we immediately obtain the following weak Gagliardo inequality, asserting that for any locally Lipschitz function $f : (M,g) \rightarrow \Real$ with $\med_\mu(f) = 0$:
\begin{equation} \label{eq:weak-Galiardo}
\int \abs{\nabla f} d\mu \geq D_{Che,N}  2^{\frac{1}{N}} \norm{f}_{L^{\frac{N}{N-1},1}(\mu)} .
\end{equation}
By Proposition \ref{prop:func-equiv}, this is equivalent to the following Nash-type inequality:
\begin{equation} \label{eq:N-weak-Gagliardo}
\norm{f}_{L^1(\mu)} \leq D_{Che,N}^{-\frac{N}{N-1}} 2^{-\frac{1}{N-1}}  (\int \abs{\nabla f} d\mu)^{\frac{N}{N-1}}  \norm{f}_{L^\infty(\mu)}^{-\frac{1}{N-1}} ~,
\end{equation}
for any locally-Lipschitz function with $\med_\mu f = 0$, which by Proposition \ref{prop:weak-Nash} with $\beta = \frac{N}{N-p}$ yields assertion (1).
Assertion (2) follows immediately from (\ref{eq:N-weak-Gagliardo}) by Proposition \ref{prop:weak-Sobolev}. 

To establish Assertion (3), note that:
\[
\norm{f}_{L^{p_1,\infty}(\mu)} \leq \norm{f}_{L^{p_1,p_1 \frac{N-1}{N}}(\mu)} \leq C_{p_1,q_1} \norm{\abs{\nabla f}}_{L^{q_1,p_1 \frac{N-1}{N}}(\mu)} , 
\]
and so for any $1$-Lipschitz function $f$ with $\med_\mu f = 0$:
\[
\sup_{t > 0}  t  \cdot \mu(\set{ \abs{f} \geq t})^{1/p_1} \leq C_{p_1,q_1} .
\]
It is well-known and immediate to verify that:
\[
\K(t) = \sup \set{ \mu( \set{ f \geq t }) \; ; \;  f \text{ is $1$-Lipschitz and } \med_\mu f = 0  } .
\]
Consequently $\K(C_{p_1,q_1} 4^{1/p_1}) \leq 1/4$, and so by invoking Part (3) of Theorem \ref{thm:Isop-Conc-NSpace}, we deduce:
\[
D_{Che,N} \geq \frac{2^{\frac{N-1}{N}}}{4^{1+1/p_1} C_{p_1,q_1}} . 
\]
Plugging this into (\ref{eq:Cpq}) and setting $\frac{1}{q(p_2)} = \frac{1}{p_2} + \frac{1}{N}$, we obtain:
\[
C_{p_2,q_2} \leq C_{p_2 , q(p_2)} \leq \frac{1}{2} 4^{1 + 1/p_1} p_2 \brac{\frac{q(p_2)}{p_2}}^{1/q(p_2)} C_{p_1,q_1} ,
\]
and the assertion now follows after expressing $q(p_2)$ in terms of $p_2$. 
\end{proof}

\subsection{Stability}

The advantage of concentration inequalities over isoperimetric ones is that they are much more robust to perturbation, and so exhibit better stability properties - such properties were obtained in our previous works \cite{EMilman-RoleOfConvexity,EMilmanGeometricApproachPartII,BartheEMilmanConservativeSpins}. Using the $\CD(0,N)$ condition and Theorem \ref{thm:Isop-Conc-NSpace}, these stability properties immediately pass to the isoperimetric level. Since the proof involves a repetition, mutatis mutandis, of the arguments from \cite{EMilman-RoleOfConvexity,EMilmanGeometricApproachPartII,BartheEMilmanConservativeSpins}, we do not repeat the computation, and only state the stability result using a non-explicit function $F_N$. 

Given two Borel probability measures $\mu_1,\mu_2$ on $(M,g)$, denote by $d_{TV}(\mu_1,\mu_2)$ their total-variation distance, by $H(\mu_2 | \mu_1)$ their relative-entropy (or Kullback--Leibler divergence), and by $W_1(\mu_1,\mu_2)$ their Wasserstein distance (we refer to \cite{EMilman-RoleOfConvexity,EMilmanGeometricApproachPartII} for definitions). 

\begin{thm} \label{thm:stability}
Assume that $(M,g,\mu_2)$ satisfies the $\CD(0,N)$ condition, $N \in (-\infty,0) \cup [n,\infty]$, and that either:
\begin{enumerate}
\item $(\int (\frac{d\mu_2}{d\mu_1})^p d\mu_1)^{1/p} \leq L< \infty$ for some $p \in (1,\infty]$  if $\mu_2 \ll \mu_1$; or
\item  $H(\mu_2 | \mu_1) \leq B < \infty$ if $N \in [n,\infty]$ and $\mu_2 \ll \mu_1$; or
\item $W_1(\mu_1,\mu_2) \leq D < \infty$ if $\frac{1}{N} > -1$; or
\item $d_{TV}(\mu_1,\mu_2) \leq 1 - \eps < 1$ and $(M,g)$ is Euclidean. 
\end{enumerate}
Then $D_{Che,N}(M,g,\mu_1) > 0$ implies $D_{Che,N}(M,g,\mu_2) \geq F_N(D_{Che,N}(M,g,\mu_1),\eps,L,p,B,D)$, for some (explicitly computable) positive function $F_N$ depending solely on its arguments and a lower bound on $\frac{1}{N}$. 
\end{thm}
\begin{proof}
Denote by $\K^i := \K(M,g,\mu_i)$ and $D_{Che,N}^i := D_{Che,N}(M,g,\mu_i)$ the corresponding concentration profiles and $N$-dimensional Cheeger constant.  
\begin{enumerate}
\item 
It was shown in \cite[Proposition 2.2]{BartheEMilmanConservativeSpins} that:
\[
(\int (\frac{d\mu_2}{d\mu_1})^p d\mu_1)^{1/p} \leq L< \infty \;\;\; \Rightarrow \;\;\; \K^2(r) \leq 2 L \K^1(r/2)^{1-1/p} . 
\]
Consequently, (\ref{eq:N-conc}) and Part (3) of Theorem \ref{thm:Isop-Conc-NSpace} yield the assertion. 
\item
This was proved for the weakest case $N=\infty$ in \cite[Theorem 5.7]{EMilmanGeometricApproachPartII}. 
\item 
Denote:
\[
C_{FM}^i := \sup \set{ \int \abs{f} d\mu_i \; ; \; \text{$f :(M,g) \rightarrow \Real$ is $1$-Lipschitz and $\med_{\mu_i} f = 0$} } .
\]
Note that the assumptions that $N \in (-\infty,-1) \cup [n,\infty]$ and $D^1_{Che,N} > 0$ ensure by (\ref{eq:N-conc}) that taking the supremum as above:
\[
C_{FM}^1 = \sup \int_0^\infty \mu_1 \set{ \abs{f} \geq r} dr \leq 2 \int_0^\infty \K_1(r) dr \leq 2 \int_0^\infty \brac{\brac{\frac{1}{2}}^{\frac{1}{N}} - \frac{r D^1_{Che,N}}{N}}_+^N dr < \infty .
\]
On the other hand, since by the Markov-Chebyshev inequality we have $\K_2(r) \leq C^2_{FM} / r$, it follows by Part (3) of Theorem \ref{thm:Isop-Conc-NSpace} that:
\[
D^2_{Che,N} = \sup_{r > 0} 2^{\frac{N-1}{N}} \frac{\frac{1}{2} -\K_2(r)}{r} \geq 2^{\frac{N-1}{N}} \frac{1}{16 C^2_{FM}} . 
\]

Now, it was shown in \cite[Lemma 5.4]{EMilmanGeometricApproachPartII} that:
\[
\abs{C^{FM}_1 - C^{FM}_2} \leq W_1(\mu_1,\mu_2) .
\]
It immediately follows from the above discussion that:
\[
D^2_{Che,N} \geq \frac{2^{\frac{N-1}{N}}}{16 C^2_{FM}} \geq \frac{2^{\frac{N-1}{N}} }{16 (C^1_{FM} + W_1(\mu_1,\mu_2))} \geq F_N(D^1_{Che,N},W_1(\mu_1,\mu_2)) . 
\]
\item
This was shown in \cite[Theorem 5.5]{EMilman-RoleOfConvexity} for the case $N=\infty$, and the proof carries over mutatis mutandis to the general case by employing the results we have obtained in this section. 
\end{enumerate}

\end{proof}

\subsection{Linear Cheeger Constant when $D < \infty$}

Finally, we have:
\begin{thm} \label{thm:linear-cheeger} 
Let $(M^n,g,\mu)$ satisfy the $\CDD(0,N,D)$ condition, with $N \in (-\infty,0] \cup [n,\infty]$ and $D < \infty$. 
Then:
\[
D_{Che,\infty}(M,g,\mu) = \inf_{v \in (0,1)} \frac{\I(M,g,\mu)(v)}{\min(v,1-v)} \geq \frac{1}{D} .
\]
\end{thm}

\begin{rem}
In the Euclidean setting, this was proved by Bobkov \cite[Section 9]{BobkovConvexHeavyTailedMeasures}, by employing the localization method. 
When $N \in [n,\infty]$, it is not difficult to improve this bound to $\frac{2}{D}$, which is best possible, as witnessed by considering the uniform measure on $[0,D]$; indeed, any $\CDD(0,N,D)$ space for $N$ in that range is in particular a $\CDD(0,\infty,D)$ space, and using the exponential model space described in (\ref{eq:exponential}) below, this amounts to an elementary computation. We do not know if the improved bound also holds when $N \in (-\infty,0]$, as the verification involves a tedious calculation which we did not pursue. 
\end{rem}

\begin{proof}
Assume first that $N < \infty$. The claim then amounts to a direct verification on the isoperimetric lower bound given by Case (3) of Corollary \ref{cor:model} (to handle $N \in (-\infty,0]$) and Case (3) of \cite[Corollary 1.4]{EMilmanSharpIsopInqsForCDD} (to handle $N \in [n,\infty)$). 
Unifying both cases, we have that in the above range of $N$, for all $v \in (0,1)$:
\[
\I(M,g,\mu)(v) \geq \inf_{\xi  > 0} \J( t^{N-1} , [\xi,\xi+D])(v) .
\]
Note that we have omitted the model space of uniform density on $[0,D]$, as it appears as the limiting case as $\xi \rightarrow \infty$. Consequently:
\[
D_{Che,\infty}(M,g,\mu) \geq \inf_{\xi > 0} D_{Che,\infty}([\xi,\xi+D],\abs{\cdot},c_{N,\xi} t^{N-1} dt) ,
\]
where $c_{N,\xi} > 0$ is a normalization constant. In general, the computation of the right-hand side above is rather tedious, and so instead we refer to a very elegant argument of Bobkov (\cite[Lemma 9.2]{BobkovConvexHeavyTailedMeasures}), who showed that any unimodal probability measure $\eta$ supported on an interval $L$ of length $D$ satisfies $D_{Che,\infty}(L,\abs{\cdot},\eta) \geq \frac{1}{D}$. As the measures $c_{N,\xi} t^{N-1} 1_{[\xi,\xi+D]}(t) dt$ are unimodal for all $N$, the result now immediately follows.

The case $N = \infty$ follows by taking the limit, or repeating the argument using Case (7) of \cite[Corollary 1.4]{EMilmanSharpIsopInqsForCDD}, which asserts that a $\CDD(0,\infty,D)$ weighted manifold satisfies:
\begin{equation} \label{eq:exponential}
\I(M^n,g,\mu) \geq \inf_{H \geq 0} \J(\exp(H t),[0,D]) ~.
\end{equation}
\end{proof}

\section{Isoperimetric, Poincar\'e and Concentration inequalities under positive curvature} \label{sec:pos}

\subsection{Poincar\'e inequality} 

Let $(M^n,g,\mu)$ satisfy the $\CD(\rho,N)$ condition with strictly positive curvature $\rho > 0$. 
When $(M,g)$ is compact and $N \in (-\infty,0)$, a Lichnerowicz-type Poincar\'e inequality was obtained by Shin-ichi Ohta \cite{Ohta-NegativeN} (for the case that $\partial M = \emptyset$) and concurrently (and independently) in our previous work with Alexander Kolesnikov \cite{KolesnikovEMilmanReillyPart1} (allowing a locally-convex boundary, i.e. having non-negative second-fundamental form):
\begin{equation} \label{eq:Poincare}
\int f^2 d\mu - (\int f d\mu)^2 \leq C_{Poin}\int \abs{\nabla f}^2 d\mu \;\;\; \text{for all smooth $f : (M,g) \rightarrow \Real$} ,
\end{equation}
where the best constant $C_{Poin} = C_{Poin}(M,g,\mu)$ in (\ref{eq:Poincare}), called the Poincar\'e constant, satisfies:
\begin{equation} \label{eq:C-Poincare}
C_{Poin} \leq \frac{1}{\rho} \frac{N-1}{N} . 
\end{equation}
This extended the Lichnerowicz estimate (the constant density case $N=n$) from the previously known range $N \in [n,\infty]$ (see \cite{KolesnikovEMilmanReillyPart1} and the references therein). 
 In \cite{KolesnikovEMilmanReillyPart1}, we also showed that the estimate (\ref{eq:C-Poincare}) is sharp, for all values of $\rho > 0$, $N \in (-\infty,-1] \cup [n,\infty]$ and $n\geq 1$. Note that the sharpness was not shown for $N \in (-1,0)$.

Up to a constant, this Poincar\'e inequality also follows from our isoperimetric analysis. Moreover, our analysis also extends to the range $[0,1)$, and shows that the constant $\frac{N-1}{N}$ in (\ref{eq:C-Poincare}) cannot be sharp as $N < 0$ increases to $0$, since given $\rho > 0$, the Poincar\'e constant $C_{Poin}$ remains uniformly bounded in $N \in [-1,1-\eps]$. Our argument is based on the celebrated Maz'ya--Cheeger inequality \cite{MazyaCheegersInq1,CheegerInq}, asserting that a linear Cheeger isoperimetric inequality always implies a Poincar\'e inequality with:
\begin{equation} \label{eq:Cheeger}
C_{Poin}(M,g,\mu) \leq \frac{4}{D^2_{Che,\infty}(M,g,\mu)} . 
\end{equation}

\begin{thm} \label{thm:positive-curvature}
Let $(M^n,g,\mu)$ satisfy the $\CD(\rho,N)$ condition with $\rho > 0$ and $N \in (-\infty,1)$. Then:
\begin{enumerate}
\item $D_{Che,\infty}(M,g,\mu) \geq \sqrt{\frac{\rho}{1-N}} \frac{1}{\int_0^\infty \cosh^{N-1}(t) dt} \geq c \sqrt{\rho} \min(1,\sqrt{1-N})$.
\item $C_{Poin}(M,g,\mu) \leq 4 \frac{1-N}{\rho} (\int_0^\infty \cosh^{N-1}(t) dt)^2 \leq \frac{4}{c^2} \frac{1}{\rho} \max(1,\frac{1}{1-N})$. 
\end{enumerate}
Under the $\CDD(0,N,D)$ condition, with $N \in (-\infty,1) \cup [n,\infty]$:
\begin{enumerate}
\setcounter{enumi}{2}
\item $C_{Poin}(M,g,\mu) \leq 4 D^2$. 
\end{enumerate}
Here $c>0$ is an (explicitly computable) numeric constant. 
\end{thm}

\begin{rem} \label{rem:Poincare}
It should be possible to show that Poincar\'e constant of $(M^n,g,\mu)$ satisfying $\CD(\rho,N)$ with $\rho > 0$ and $N \in (-\infty,1) \cup [n,\infty]$ is always majorized by the Poincar\'e  constant of the one-dimensional model space from Case (1) of Corollary \ref{cor:model}, having density $\cosh^{N-1}(\sqrt{-\delta}t)$ as in (\ref{eq:cosh-model}) below. This is indeed the case when $N \in (-\infty,-1] \cup [n,\infty]$, since in that range, it was verified in \cite{KolesnikovEMilmanReillyPart1} that the Poincar\'e constant of our model space is indeed $\frac{1}{\rho} \frac{N-1}{N}$; as explained above, this is no longer the case when $N$ approaches $0$. We mention that in \cite{KlartagLocalizationOnManifolds}, Klartag has already shown that the Poincar\'e constant of $(M^n,g,\mu)$ satisfying $\CDD(\rho,N,D)$  is always majorized by the Poincar\'e constant of the ``worst" one-dimensional density satisfying the $\CDD(\rho,N,D)$ condition, so it remains to establish that our model density is indeed the worst one when $\rho > 0$ and $D=\infty$; we do not pursue this direction here.
\end{rem}

\begin{proof}[Proof of Theorem \ref{thm:positive-curvature}]
By Case (1) of Corollary \ref{cor:model}:
\begin{equation} \label{eq:cosh-model}
D_{Che,\infty}(M,g,\mu) \geq D := D_{Che,\infty}(\Real,\abs{\cdot}, C^{-1}_{N,\rho} \cosh^{N-1}(\sqrt{-\delta}t) dt) \;\; , \;\; \delta = \frac{\rho}{N-1} ,
\end{equation}
where $C_{N,\rho} = \int_{-\infty}^\infty \cosh^{N-1}(\sqrt{-\delta}t) dt$ is a normalization constant. Since the density $\cosh^{N-1}(\sqrt{-\delta} t)$ is log-concave when $N < 1$, a result of Bobkov \cite{BobkovExtremalHalfSpaces} for measures on $\Real$ asserts that $\I = \I(\cosh^{N-1}(\sqrt{-\delta}t), \Real) = \J(\cosh^{N-1}(\sqrt{-\delta}t), \Real)$ is concave, and hence $D = \inf_{v \in (0,1)} \I(v) / \min(v,1-v)$ is attained at $v=1/2$ (alternatively, for a self-contained argument, use the weak concavity of $\I$ asserted in Part (1) of Theorem \ref{thm:Isop-Conc-NSpace}). Consequently:
\[
D_{Che,\infty}(M,g,\mu) \geq D = 2 \I(1/2) =  \frac{2}{C_{N,\rho}} = \frac{\sqrt{-\delta}}{\int_0^\infty \cosh^{N-1}(t) dt} ,
\]
and the first assertion is proved; the estimate on $\int_0^\infty \cosh^{N-1}(t) dt$ is proved in Lemma \ref{lem:cosh-calc} below. The second assertion follows immediately by the Maz'ya--Cheeger inequality (\ref{eq:Cheeger}). The third assertion follows by Theorem \ref{thm:linear-cheeger} coupled with (\ref{eq:Cheeger}). 
\end{proof}

To conclude the proof and for the subsequent analysis, we require the following:

\begin{lem} \label{lem:cosh-calc}
\begin{enumerate}
The following estimates hold:
\item For all $t \geq 0$:
\[
\exp\brac{\frac{\min(t^2/2,t)}{\cosh^2 2}} \leq \cosh(t) \leq \exp(\min(t^2 / 2 , t)) . \]
\item For $N < 1$:
\[
\max\brac{\sqrt{\frac{\pi}{2}} \frac{1}{\sqrt{1-N}},\frac{1}{1-N}} \leq \int_0^\infty \cosh^{N-1}(t) dt \leq \sqrt{\frac{\pi}{2}} \frac{\cosh 2}{\sqrt{1-N}} +\frac{\cosh^2 2}{1-N} .
\]
\end{enumerate}
\end{lem}
\begin{proof}
The estimate $\cosh(t) \leq \exp(t)$ is obvious, and $\cosh(t) \leq \exp(t^2 / 2)$ is verified by comparing coefficients of the corresponding Taylor series. 
As for the inequality in the other direction, observe that:
\[
(\log \cosh)''(t) = \frac{1}{\cosh^2 t} \geq \begin{cases} \frac{1}{\cosh^2 2} & t \in [0,2]\\ 0 & t \in [2,\infty)  \end{cases} ,
\]
and so the lower bound on $\cosh(t)$ follows by integrating this function twice. The second assertion immediately follows from the first. 
\end{proof}

\subsection{Two-level Behaviour}

A weighted manifold satisfying $\CD(\rho,N)$ with $\rho > 0$ and $N \in (-\infty,1)$ has very interesting concentration properties, as described below. 

\begin{prop} \label{prop:two-level-conc}
Let $(M^n,g,\mu)$ satisfy $\CD(\rho,N)$ with $\rho > 0$ and $N \in (-\infty,1)$, let $\K = \K(M,g,\mu)$ denote its concentration profile. Set $\delta = \frac{\rho}{N-1}$, and let $\K_0 = \K(\Real,\abs{\cdot},C^{-1}_{N,\rho} \cosh^{N-1}(\sqrt{-\delta} t) dt)$ denote the concentration profile of our model space. Then for any $r> 0$:
\[
\K(r) \leq \K_0(r)  = \frac{\int_{\sqrt{-\delta} r}^\infty \cosh^{N-1}(t) dt}{2 \int_0^\infty \cosh^{N-1}(t) dt } \leq 
\begin{cases}   
C \min(1,\sqrt{1-N}) \frac{\exp(- c \rho r^2)}{1 + \sqrt{\rho} r }  & r \in \left [ 0,\sqrt{\frac{1-N}{\rho}} \right ] \\ C \min(1, \frac{1}{\sqrt{1-N}}) \exp(- c \sqrt{1-N} \sqrt{\rho} r) & \text{otherwise} 
\end{cases} ,
\]
where $c , C > 0$ are numeric constants. 
\end{prop}

Certainly, the exponential decay of $\K(r)$ when $r \rightarrow \infty$ is expected, since a result of Gromov and V. Milman \cite{GromovMilmanLevyFamilies} asserts that a Poincar\'e inequality always implies this type of tail-decay. However, the Gaussian-type decay for $r \in [0,\sqrt{\frac{1-N}{\rho}}]$ is somewhat surprising. 

\begin{proof}
The inequality $\K(r) \leq \K_0(r)$ follows since the isoperimetric minimizers of our model-space are nested half-lines. Indeed, 
denoting by $\I$ and $\I_0$ the isoperimetric profiles of $(M,g,\mu)$ and the model-space, respectively. Denoting the model density $C_{N,\rho}^{-1} \cosh^{N-1}(\sqrt{-\delta} t)$ by $f$ and setting $F(r) = \int_r^\infty f(t) dt$, we have by (\ref{eq:isop-conc}):
\[
\K^{-1}(v) \leq \int_v^{1/2} \frac{dv}{\I(v)} \leq \int_v^{1/2} \frac{dv}{\I_0(v)} = \int_v^{1/2} \frac{dv}{f \circ F^{-1}(v)} .
\]
Performing the change of variables $v = F(r)$, we see that $\K^{-1}(v) \leq \int_0^{F^{-1}(v)} dr = F^{-1}(v)$, and hence $\K(r) \leq F(r)$. In particular, we have $\K_0(r) = F(r)$, and the first inequality follows. The second inequality follows from the estimates of Lemma \ref{lem:cosh-calc} and the standard (rough) Gaussian tail estimate:
\[
\int_t^\infty \exp(-s^2/2) ds \leq \frac{C \exp(-t^2/2)}{1 + t}  \;\;\; \forall t > 0 . 
\]
\end{proof}

See Subsection \ref{subsec:beyond} for a further discussion of this two-level behaviour.

\section{Concluding Remarks} \label{sec:conclude}

\subsection{Beyond Poincar\'e under positive curvature} \label{subsec:beyond}

It would be interesting to devise a natural functional inequality which captures the two-level behaviour of positively curved spaces of dimension $N < 1$, described in the previous section. Certainly, such a functional inequality cannot imply concentration stronger than exponential, as witnessed by our model density $\cosh^{N-1}(\sqrt{-\delta} t)$, and so contrary to the case when $N \in [n,\infty]$, such spaces do not satisfy in general a log-Sobolev inequality (which implies by the Herbst argument sub-Gaussian concentration, see \cite{Ledoux-Book}). However, a mixture of Poincar\'e for large-deviations and log-Sobolev for small ones is quite possible, and as we saw in the previous section, quite natural.

Under the Entropic Curvature-Dimension condition $\CD^e(\rho,N)$, which was shown by Erbar--Kuwada--Sturm \cite{EKS-Equivalence} to be equivalent when $N > 0$ to the usual $\CD(\rho,N)$ condition for essentially non-branching spaces (such as weighted-manifolds, Finsler and Alexandrov spaces), Ohta obtained in \cite{Ohta-NegativeN} variants for $N < 0$ and $\rho > 0$ of the HWI, Talagrand and log-Sobolev inequalities. However, when $N<0$, it is not clear whether the $\CD^e(\rho,N)$ condition is equivalent to the $\CD(\rho,N)$ one (Ohta showed that the former only implies the latter).

\subsection{Additional Properties of $\CD(\rho,N)$ spaces with $N < 0$}

Various additional properties of $\CD(\rho,N)$ weighted-manifolds with $N<0$ have been obtained by Ohta in \cite{Ohta-NegativeN} and Klartag in \cite{KlartagLocalizationOnManifolds}. We do not present a full account here, but only mention two results: a Brunn--Minkowski inequality verified by Ohta for all $N < 0$, extending previous results of Sturm \cite{SturmCD12} and Lott--Villani \cite{LottVillaniGeneralizedRicci} for the case $N \in [n,\infty]$, which are a particular case of a very general Brunn--Minkowski inequality of Cordero-Erausquin--McCann--Schmuckenschl{\"a}ger \cite{CMSInventiones, CMSManifoldWithDensity} involving distortion coefficients; and as already described in Remark \ref{rem:Poincare}, a reduction by Klartag of the Poincar\'e inequality to the one-dimensional case when $N < 1$, extending a previous result by Bakry and Qian \cite{BakryQianGenRicComparisonThms} for the case $N \in [n,\infty]$.

\subsection{Alternative Derivation} \label{subsec:Klartag}

Concurrently to our work, Bo'az Klartag has devised in \cite{KlartagLocalizationOnManifolds} a remarkable alternative method for reducing isoperimetric and functional inequalities to the one-dimensional case, by extending the localization method of Payne--Weinberger, Gromov--V. Milman and Kannan--Lov\'asz--Simonovits, from a linear setting to an arbitrary Riemannian one (see \cite{KlartagLocalizationOnManifolds} and the references therein). In particular, under a $\CD(\rho,N)$ condition, Klartag reduces the isoperimetric problem to the study of one-dimensional densities satisfying $\CD(\rho,N)$, namely the second-order differential inequality (\ref{eq:1d-ODE}).
One advantage of Klartag's method is that he does not need to rely on the deep regularity results provided by Geometric Measure Theory, on which our entire approach is based. Another advantage is that his method easily adapts to the study of functional inequalities as well, and in general seems more flexible. On the other hand, the fact that we can directly work with an isoperimetric minimizer, whose regular part of the boundary already has constant (generalized) mean-curvature, allows us to directly reduce our sought-after isoperimetric inequality to that on our one-dimensional model densities, characterized by the equality case in (\ref{eq:1d-ODE}), thus avoiding any further analysis of the one-dimensional case. 

\subsection{The case $N \in [1,n)$} \label{subsec:no-extend}

It is easy to see that Theorem \ref{thm:CDD-II}, asserting an isoperimetric inequality on $(M^n,g,\mu)$ satisfying the $\CDD(\rho,N,D)$ condition, cannot be further extended (at least, as is) to the range $N \in [1,n)$, and so our extension in this work to the entire range $(-\infty,1) \cup [n,\infty]$ is best possible. To see this, note the $\CD(\rho,N)$ definition is monotone in $\frac{1}{N-n}$, so that:
\begin{equation} \label{eq:paradox}
\frac{1}{N_1-n} \geq \frac{1}{N_2 - n} \;\; \Rightarrow \;\;  \CD(\rho,N_1) \text{ implies } \CD(\rho,N_2) . 
\end{equation}
Given $N \in [1,n)$, consider the construction from \cite[Subsection 3.2]{KolesnikovEMilmanReillyPart1}, which emulates the model measure on $\Real$ for the $\CD(1,1-\eps)$ condition, namely $\cosh(t / \sqrt{\eps})^{-\eps} dt$, on an $n$-dimensional weighted manifold $(M^n,g,\mu)$. By (\ref{eq:paradox}), this weighted manifold also satisfies the $\CD(1,N)$ condition, despite having arbitrarily bad isoperimetric and concentration properties, as witnessed by letting $\eps \rightarrow 0$ in the estimates of Proposition \ref{prop:two-level-conc}. This is in sharp contrast to the case when $N \in [n,\infty]$, wherein weighted manifolds satisfying $\CD(1,N)$ have isoperimetric and concentration properties which are at least as good as the scaled $N$-sphere $\sqrt{N-1} S^N$ ($N \in [n,\infty)$) or Gaussian measure ($N = \infty$),  by the theorems of Gromov--L\'evy, Bayle and Bakry--Ledoux mentioned in the Introduction; in particular, they all satisfy a sub-Gaussian concentration ($\K(r) \leq \exp(-c r^2)$).

\subsection{Relation of $\CD(\rho,N)$ to Bakry--\'Emery's original definition} \label{subsec:BE}

In \cite{BakryEmery}, Bakry and \'Emery originally defined the Curvature-Dimension condition in the context of abstract diffusion generators. Given an appropriate diffusion generator $L$, they defined the associated $\Gamma$ and $\Gamma_2$ operators, and defined the Curvature-Dimension condition, which we denote by $\text{BE}(\rho,N)$, as the property that:
\begin{equation} \label{eq:BE}
\Gamma_2(f) \geq \rho \Gamma(f) + \frac{1}{N} (Lf)^2 ,
\end{equation}
for all test functions $f$. When $L = \Delta_g - \scalar{\nabla_g , \nabla_g V}$ on a Riemannian manifold $(M,g,\mu = \exp(-V) \vol_g)$, where $\Delta_g$ denotes the Laplace-Beltrami operator, this condition translates to:
\[
\Ric_{g,\mu}(\nabla f, \nabla f) + \norm{\text{Hess}_g f}^2 \geq \rho \abs{\nabla f}^2 + \frac{1}{N}(L f)^2 . 
\]
It was shown by Bakry \cite[Section 6]{BakryStFlour} for $N \in [n,\infty]$, and extended to all $N \in (-\infty,0) \cup [n,\infty]$ in \cite[Remark 2.4]{KolesnikovEMilmanReillyPart1}, that in this case, $\text{BE}(\rho,N)$ is equivalent to $\CD(\rho,N)$ for all $\rho \in \Real$. However, we note that the elementary argument used to deduce this equivalence, based on the Cauchy--Schwarz inequality, no longer applies for $N \in [0,n)$, and so it seems that the two definitions diverge in that range. Of particular interest is the range $N \in (0,1)$, which we have developed in this work for the $\CD(\rho,N)$ definition. The difference between the two definitions is especially evident from the fact (\ref{eq:paradox}) that the $\CD(\rho,N)$ condition is monotone in $\frac{1}{N-n}$, whereas the $\text{BE}(\rho,N)$ condition (\ref{eq:BE}) is clearly monotone in $\frac{1}{N}$.

\subsection{Future Work}

In general, the $\CD(\rho,N)$ condition will not yield any meaningful information on $(M^n,g,\mu)$ when $N \in [1, n)$. However, 
in a subsequent work \cite{EMilman-GradedCD}, we devise a more restrictive condition we dub the Graded Curvature-Dimension condition, which in some cases allows handling the latter regime. Even in the classical regime $N \in [n,\infty)$, we may use this condition to sharpen our previous isoperimetric results from \cite{EMilmanSharpIsopInqsForCDD} \text{in the Euclidean setting} (it was shown in \cite{EMilmanSharpIsopInqsForCDD} that our isoperimetric inequalities are sharp in the Riemannian setting, but when $\rho \neq 0$, we show in \cite{EMilman-GradedCD} that they are no longer sharp in the Euclidean one).

Another direction which is worth looking into, is investigating whether our results from \cite{EMilmanGeometricApproachPartI}, asserting the equivalence of concentration and isoperimetric inequalities on weighted manifolds satisfying the $\CD(\rho,N)$ condition with $\rho \leq 0$ and $N = \infty$, may be extended to the range $N \in (-\infty,1)$.

\setlinespacing{1.0}
\setlength{\bibspacing}{0pt}

\bibliographystyle{plain}
\bibliography{../../../ConvexBib}

\def\cprime{$'$} \def\textasciitilde{$\sim$}
\begin{thebibliography}{10}

\bibitem{BakryStFlour}
D.~Bakry.
\newblock L'hypercontractivit\'e et son utilisation en th\'eorie des
  semigroupes.
\newblock In {\em Lectures on probability theory ({S}aint-{F}lour, 1992)},
  volume 1581 of {\em Lecture Notes in Math.}, pages 1--114. Springer, Berlin,
  1994.

\bibitem{BakryEmery}
D.~Bakry and M.~{\'E}mery.
\newblock Diffusions hypercontractives.
\newblock In {\em S\'eminaire de probabilit\'es, XIX, 1983/84}, volume 1123 of
  {\em Lecture Notes in Math.}, pages 177--206. Springer, Berlin, 1985.

\bibitem{BakryLedoux}
D.~Bakry and M.~Ledoux.
\newblock L\'evy-{G}romov's isoperimetric inequality for an
  infinite-dimensional diffusion generator.
\newblock {\em Invent. Math.}, 123(2):259--281, 1996.

\bibitem{BakryQianGenRicComparisonThms}
D.~Bakry and Z.~Qian.
\newblock Volume comparison theorems without {J}acobi fields.
\newblock In {\em Current trends in potential theory}, volume~4 of {\em Theta
  Ser. Adv. Math.}, pages 115--122. Theta, Bucharest, 2005.

\bibitem{BartheEMilmanConservativeSpins}
F.~Barthe and E.~Milman.
\newblock Transference principles for log-{S}obolev and spectral-gap with
  applications to conservative spin systems.
\newblock {\em Comm. Math. Phys.}, 323(2):575--625, 2013.

\bibitem{BavardPansu}
C.~Bavard and P.~Pansu.
\newblock Sur le volume minimal de $\bold {R}\sp 2$.
\newblock {\em Ann. Sci. \'Ecole Norm. Sup.}, 19(4):479--490, 1986.

\bibitem{BayleThesis}
V.~Bayle.
\newblock {\em Propri\'et\'es de concavit\'e du profil isop\'erim\'etrique et
  applications}.
\newblock PhD thesis, Institut Joseph Fourier, Grenoble, 2004.

\bibitem{BobkovExtremalHalfSpaces}
S.~Bobkov.
\newblock Extremal properties of half-spaces for log-concave distributions.
\newblock {\em Ann. Probab.}, 24(1):35--48, 1996.

\bibitem{BobkovConvexHeavyTailedMeasures}
S.~G. Bobkov.
\newblock Large deviations and isoperimetry over convex probability measures
  with heavy tails.
\newblock {\em Electron. J. Probab.}, 12:1072--1100 (electronic), 2007.

\bibitem{BobkovHoudreMemoirs}
S.~G. Bobkov and C.~Houdr{\'e}.
\newblock Some connections between isoperimetric and {S}obolev-type
  inequalities.
\newblock {\em Mem. Amer. Math. Soc.}, 129(616), 1997.

\bibitem{BobkovLedouxWeightedPoincareForHeavyTails}
S.~G. Bobkov and M.~Ledoux.
\newblock Weighted {P}oincar\'e-type inequalities for {C}auchy and other convex
  measures.
\newblock {\em Ann. Probab.}, 37(2):403--427, 2009.

\bibitem{BorellConvexMeasures}
Ch. Borell.
\newblock Convex set functions in {$d$}-space.
\newblock {\em Period. Math. Hungar.}, 6(2):111--136, 1975.

\bibitem{BLM-Book}
S.~Boucheron, G.~Lugosi, and P.~Massart.
\newblock {\em Concentration Inequalities: A Nonasymptotic Theory of
  Independence}.
\newblock Oxford University Press, 2013.

\bibitem{BrascampLiebPLandLambda1}
H.~J. Brascamp and E.~H. Lieb.
\newblock On extensions of the {B}runn-{M}inkowski and {P}r\'ekopa-{L}eindler
  theorems, including inequalities for log concave functions, and with an
  application to the diffusion equation.
\newblock {\em J. Func. Anal.}, 22(4):366--389, 1976.

\bibitem{ChavelEigenvalues}
I.~Chavel.
\newblock {\em Eigenvalues in {R}iemannian geometry}, volume 115 of {\em Pure
  and Applied Mathematics}.
\newblock Academic Press Inc., Orlando, FL, 1984.

\bibitem{ChavelRiemannianGeometry1stEd}
I.~Chavel.
\newblock {\em Riemannian geometry---a modern introduction}, volume 108 of {\em
  Cambridge Tracts in Mathematics}.
\newblock Cambridge University Press, Cambridge, 1993.

\bibitem{CheegerInq}
J.~Cheeger.
\newblock A lower bound for the smallest eigenvalue of the {L}aplacian.
\newblock In {\em Problems in analysis (Papers dedicated to Salomon Bochner,
  1969)}, pages 195--199. Princeton Univ. Press, Princeton, N. J., 1970.

\bibitem{CMSInventiones}
D.~Cordero-Erausquin, R.~J. McCann, and M.~Schmuckenschl{\"a}ger.
\newblock A {R}iemannian interpolation inequality \`a la {B}orell, {B}rascamp
  and {L}ieb.
\newblock {\em Invent. Math.}, 146(2):219--257, 2001.

\bibitem{CMSManifoldWithDensity}
D.~Cordero-Erausquin, R.~J. McCann, and M.~Schmuckenschl{\"a}ger.
\newblock Pr\'ekopa-{L}eindler type inequalities on {R}iemannian manifolds,
  {J}acobi fields, and optimal transport.
\newblock {\em Ann. Fac. Sci. Toulouse Math. (6)}, 15(4):613--635, 2006.

\bibitem{EKS-Equivalence}
M.~Erbar, K.~Kuwada, and Sturm K.-T.
\newblock On the equivalence of the entropic {C}urvature-{D}imension condition
  and {B}ochner's inequality on metric measure spaces.
\newblock manuscript, arxiv.org/abs/1303.4382, 2013.

\bibitem{FedererFleming}
H.~Federer and W.~H. Fleming.
\newblock Normal and integral currents.
\newblock {\em Ann. of Math. (2)}, 72:458--520, 1960.

\bibitem{GHLBookEdition3}
S.~Gallot, D.~Hulin, and J.~Lafontaine.
\newblock {\em Riemannian geometry}.
\newblock Universitext. Springer-Verlag, Berlin, third edition, 2004.

\bibitem{GalingInequalitiesBook}
D.~J.~H. Garling.
\newblock {\em Inequalities: a journey into linear analysis}.
\newblock Cambridge University Press, Cambridge, 2007.

\bibitem{GrafakosClassicalFourierAnalysis2ndEd}
L.~Grafakos.
\newblock {\em Classical {F}ourier analysis}, volume 249 of {\em Graduate Texts
  in Mathematics}.
\newblock Springer, New York, second edition, 2008.

\bibitem{GromovGeneralizationOfLevy}
M.~Gromov.
\newblock {P}aul {L}\'evy's isoperimetric inequality.
\newblock preprint, I.H.E.S., 1980.

\bibitem{GromovMilmanLevyFamilies}
M.~Gromov and V.~D. Milman.
\newblock A topological application of the isoperimetric inequality.
\newblock {\em Amer. J. Math.}, 105(4):843--854, 1983.

\bibitem{HeintzeKarcher}
E.~Heintze and H.~Karcher.
\newblock A general comparison theorem with applications to volume estimates
  for submanifolds.
\newblock {\em Ann. Sci. \'Ecole Norm. Sup. (4)}, 11(4):451--470, 1978.

\bibitem{KLS}
R.~Kannan, L.~Lov{\'a}sz, and M.~Simonovits.
\newblock Isoperimetric problems for convex bodies and a localization lemma.
\newblock {\em Discrete Comput. Geom.}, 13(3-4):541--559, 1995.

\bibitem{KennardWylie-WeightedSectionalCurvature}
L.~Kennard and W.~Wylie.
\newblock Positive weighted sectional curvature.
\newblock arXiv:1410.1558, 2014.

\bibitem{KlartagLocalizationOnManifolds}
B.~Klartag.
\newblock Needle decompositions in {R}iemannian geometry.
\newblock To appear in Mem. Amer. Math. Soc., arXiv:1408.6322.

\bibitem{KolesnikovEMilmanReillyPart1}
A.~V. Kolesnikov and E.~Milman.
\newblock {B}rascamp--{L}ieb-type inequalities on weighted {R}iemannian
  manifolds with boundary.
\newblock To appear in J. Geom. Anal., arXiv:1310.2526.

\bibitem{Kuwert}
E.~Kuwert.
\newblock Note on the isoperimetric profile of a convex body.
\newblock In {\em Geometric analysis and nonlinear partial differential
  equations}, pages 195--200. Springer, Berlin, 2003.

\bibitem{LedouxLectureNotesOnDiffusion}
M.~Ledoux.
\newblock The geometry of {M}arkov diffusion generators.
\newblock {\em Ann. Fac. Sci. Toulouse Math. (6)}, 9(2):305--366, 2000.

\bibitem{Ledoux-Book}
M.~Ledoux.
\newblock {\em The concentration of measure phenomenon}, volume~89 of {\em
  Mathematical Surveys and Monographs}.
\newblock American Mathematical Society, Providence, RI, 2001.

\bibitem{LedouxTalagrand-Book}
M.~Ledoux and M.~Talagrand.
\newblock {\em Probability in {B}anach spaces}, volume~23 of {\em Ergebnisse
  der Mathematik und ihrer Grenzgebiete (3) [Results in Mathematics and Related
  Areas (3)]}.
\newblock Springer-Verlag, Berlin, 1991.
\newblock Isoperimetry and processes.

\bibitem{Lichnerowicz1970GenRicciTensorCRAS}
A.~Lichnerowicz.
\newblock Vari\'et\'es riemanniennes \`a tenseur {C} non n\'egatif.
\newblock {\em C. R. Acad. Sci. Paris S\'er. A-B}, 271:A650--A653, 1970.

\bibitem{Lichnerowicz1970GenRicciTensor}
A.~Lichnerowicz.
\newblock Vari\'et\'es k\"ahl\'eriennes \`a premi\`ere classe de {C}hern non
  negative et vari\'et\'es riemanniennes \`a courbure de {R}icci
  g\'en\'eralis\'ee non negative.
\newblock {\em J. Differential Geometry}, 6:47--94, 1971/72.

\bibitem{LottRicciTensorProperties}
J.~Lott.
\newblock Some geometric properties of the {B}akry-\'{E}mery-{R}icci tensor.
\newblock {\em Comment. Math. Helv.}, 78(4):865--883, 2003.

\bibitem{LottVillaniGeneralizedRicci}
J.~Lott and C.~Villani.
\newblock Ricci curvature for metric-measure spaces via optimal transport.
\newblock {\em Ann. of Math. (2)}, 169(3):903--991, 2009.

\bibitem{MazyaSobolevImbedding}
V.~G. Maz{\cprime}ja.
\newblock Classes of domains and imbedding theorems for function spaces.
\newblock {\em Dokl. Acad. Nauk SSSR}, 3:527--530, 1960.
\newblock Engl. transl. Soviet Math. Dokl., 1 (1961) 882--885.

\bibitem{MazyaCheegersInq1}
V.~G. Maz{\cprime}ja.
\newblock The negative spectrum of the higher-dimensional {S}chr\"odinger
  operator.
\newblock {\em Dokl. Akad. Nauk SSSR}, 144:721--722, 1962.
\newblock Engl. transl. Soviet Math. Dokl., 3 (1962) 808--810.

\bibitem{EMilmanHarmonicMeasures}
E.~Milman.
\newblock Harmonic measures on the sphere via curvature-dimension.
\newblock To appear in Annales de la Facult\'e des Sciences de Toulouse,
  arXiv:1505.04335.

\bibitem{EMilman-RoleOfConvexity}
E.~Milman.
\newblock On the role of convexity in isoperimetry, spectral-gap and
  concentration.
\newblock {\em Invent. Math.}, 177(1):1--43, 2009.

\bibitem{EMilmanGeometricApproachPartI}
E.~Milman.
\newblock Isoperimetric and concentration inequalities - equivalence under
  curvature lower bound.
\newblock {\em Duke Math. J.}, 154(2):207--239, 2010.

\bibitem{EMilmanIsoperimetricBoundsOnManifolds}
E.~Milman.
\newblock Isoperimetric bounds on convex manifolds.
\newblock In C.~Houdr\'e, M.~Ledoux, E.~Milman, and M.~Milman, editors, {\em
  Concentration, Functional Inequalities and Isoperimetry}, volume 545 of {\em
  Contemporary Mathematics}, pages 195--208. Amer. Math. Soc., 2011.

\bibitem{EMilmanGeometricApproachPartII}
E.~Milman.
\newblock Properties of isoperimetric, functional and transport-entropy
  inequalities via concentration.
\newblock {\em Probab. Theory Relat. Fields}, 152:475--507, 2012.

\bibitem{EMilman-GradedCD}
E.~Milman.
\newblock Beyond traditional {C}urvature-{D}imension {I}{I}: {G}raded
  {C}urvature-{D}imension condition and applications.
\newblock manuscript, 2015.

\bibitem{EMilmanSharpIsopInqsForCDD}
E.~Milman.
\newblock Sharp isoperimetric inequalities and model spaces for the
  curvature-dimension-diameter condition.
\newblock {\em J. Eur. Math. Soc. (JEMS)}, 17(5):1041--1078, 2015.

\bibitem{EMilmanRotemHomogeneous}
E.~Milman and L.~Rotem.
\newblock Complemented {B}runn-{M}inkowski inequalities and isoperimetry for
  homogeneous and non-homogeneous measures.
\newblock {\em Adv. Math.}, 262:867--908, 2014.

\bibitem{MorganManifoldsWithDensity}
F.~Morgan.
\newblock Manifolds with density.
\newblock {\em Notices Amer. Math. Soc.}, 52(8):853--858, 2005.

\bibitem{MorganBook4Ed}
F.~Morgan.
\newblock {\em Geometric measure theory (a beginner's guide)}.
\newblock Elsevier/Academic Press, Amsterdam, fourth edition, 2009.

\bibitem{MorganJohnson}
F.~Morgan and D.~L. Johnson.
\newblock Some sharp isoperimetric theorems for {R}iemannian manifolds.
\newblock {\em Indiana Univ. Math. J.}, 49(3):1017--1041, 2000.

\bibitem{NguyenDimensionalBrascampLieb}
V.~H. Nguyen.
\newblock Dimensional variance inequalities of {B}rascamp-{L}ieb type and a
  local approach to dimensional {P}r\'ekopa's theorem.
\newblock {\em J. Funct. Anal.}, 266(2):931--955, 2014.

\bibitem{OhtaTakatsu-GenEntropyI}
S.~Ohta and A.~Takatsu.
\newblock Displacement convexity of generalized relative entropies.
\newblock {\em Adv. Math.}, 228(3):1742--1787, 2011.

\bibitem{OhtaTakatsu-GenEntropyII}
S.~Ohta and A.~Takatsu.
\newblock Displacement convexity of generalized relative entropies. {II}.
\newblock {\em Comm. Anal. Geom.}, 21(4):687--785, 2013.

\bibitem{Ohta-NegativeN}
S.-I. Ohta.
\newblock {$(K,N)$}-convexity and the curvature-dimension condition for
  negative {$N$}.
\newblock {\em J. Geom. Anal.}, 26(3):2067--2096, 2016.

\bibitem{PerelmanEntropyFormulaForRicciFlow}
G.~Perelman.
\newblock The entropy formula for the {R}icci flow and its geometric
  applications.
\newblock arxiv.org/abs/math/0211159, 2002.

\bibitem{QianWeightedVolumeThms}
Z.~Qian.
\newblock Estimates for weighted volumes and applications.
\newblock {\em Quart. J. Math. Oxford Ser. (2)}, 48(190):235--242, 1997.

\bibitem{SternbergZumbrun}
P.~Sternberg and K.~Zumbrun.
\newblock On the connectivity of boundaries of sets minimizing perimeter
  subject to a volume constraint.
\newblock {\em Comm. Anal. Geom.}, 7(1):199--220, 1999.

\bibitem{SturmCD12}
K.-T. Sturm.
\newblock On the geometry of metric measure spaces. {I} and {II}.
\newblock {\em Acta Math.}, 196(1):65--177, 2006.

\bibitem{VonRenesseSturmRicciChar}
M.-K. von Renesse and K.-T. Sturm.
\newblock Transport inequalities, gradient estimates, entropy, and {R}icci
  curvature.
\newblock {\em Comm. Pure Appl. Math.}, 58(7):923--940, 2005.

\bibitem{WeiWylie-GenRicciTensor}
G.~Wei and W.~Wylie.
\newblock Comparison geometry for the {B}akry-{E}mery {R}icci tensor.
\newblock {\em J. Differential Geom.}, 83(2):377--405, 2009.

\bibitem{Wylie-SectionalCurvature}
W.~Wylie.
\newblock Sectional curvature for riemannian manifolds with density.
\newblock arXiv:1311.0267; to appear in Geometriae Dedicata, 2013.

\bibitem{Wylie-CheegerGromoll}
W.~Wylie.
\newblock A warped product version of the cheeger-gromoll splitting theorem.
\newblock arXiv:1506.03800, 2015.

\end{thebibliography}

\end{document}